\documentclass[11pt, twoside, leqno]{article}

\usepackage{graphics}
\usepackage{graphicx}
\usepackage{amssymb}
\usepackage{amsmath}
\usepackage{amsthm}
\usepackage{color}
\usepackage{mathrsfs}

\usepackage{indentfirst}

\usepackage{txfonts}

\allowdisplaybreaks

\def\red{\color{red}}

\def\rr{{\mathbb R}}
\def\rn{{\mathbb{R}^n}}

\def\zz{{\mathbb Z}}
\def\cc{{\mathbb C}}

\def\nn{{\mathbb N}}

\def\cf{{\mathcal F}}
\def\cg{{\mathcal G}}

\def\cm{{\mathcal M}}

\def\fz{\infty }

\def\dz{\delta}

\def\ez{\epsilon}

\def\lz{\lambda}

\def\oz{{\omega}}

\def\tz{\theta}

\def\lf{\left}
\def\r{\right}

\def\ls{\lesssim}

\def\noz{\nonumber}

\def\BMO{\mathop\mathrm{\,BMO\,}}

\def\loc{{\mathop\mathrm{\,loc\,}}}
\def\supp{\mathop\mathrm{\,supp\,}}

\def\XXint#1#2#3{{\setbox0=\hbox{$#1{#2#3}{\int}$ }
\vcenter{\hbox{$#2#3$ }}\kern-.6\wd0}}

\DeclareMathOperator{\esssup}{ess\,sup}
\DeclareMathOperator{\essinf}{ess\,inf}

\def\({\left(}
\def \){ \right)}

\def\lz{{\lambda}}

\def\BB{{\mathbb B}}

\newtheorem{theorem}{Theorem}[section]
\newtheorem{lemma}[theorem]{Lemma}

\newtheorem{assumption}[theorem]{Assumption}
\newtheorem{proposition}[theorem]{Proposition}

\theoremstyle{definition}
\newtheorem{remark}[theorem]{Remark}
\newtheorem{definition}[theorem]{Definition}
\renewcommand{\appendix}{\par
   \setcounter{section}{0}%
   \setcounter{subsection}{0}%
   \setcounter{subsubsection}{0}%
   \gdef\thesection{\@Alph\c@section}%
   \gdef\thesubsection{\@Alph\c@section.\@arabic\c@subsection}%
   \gdef\theHsection{\@Alph\c@section.}%
   \gdef\theHsubsection{\@Alph\c@section.\@arabic\c@subsection}%
   \csname appendixmore\endcsname
 }

\numberwithin{equation}{section}

\begin{document}

\title{\bf\Large Compactness Characterizations of
Commutators on Ball Banach Function Spaces
\footnotetext{\hspace{-0.35cm} 2020 {\it
Mathematics Subject Classification}. Primary 47B47;
Secondary 42B20, 42B25, 42B30, 42B35, 46E30.
\endgraf {\it Key words and phrases.} ball Banach function space, commutator,
convolutional singular integral operator,
BMO, CMO, extrapolation, Fr\'{e}chet--Kolmogorov theorem.
\endgraf This project is supported by the National
Natural Science Foundation of China (Grant Nos. 11971058, 12071197 and 11871100)
and the National Key Research
and Development Program of China
(Grant No. 2020YFA0712900).}}
\author{Jin Tao, Dachun Yang\footnote{Corresponding author,
E-mail: \texttt{dcyang@bnu.edu.cn}/{\red January 19, 2021}/Final version.},
\ Wen Yuan and Yangyang Zhang}
\date{}
\maketitle

\vspace{-0.7cm}

\begin{center}
\begin{minipage}{13cm}
{\small {\bf Abstract}\quad
Let $X$ be a ball Banach function space on ${\mathbb R}^n$. Let $\Omega$ be
a Lipschitz function on the unit sphere of ${\mathbb R}^n$,
which is homogeneous of degree zero and has mean value zero, and
let $T_\Omega$ be the convolutional singular integral operator
with kernel $\Omega(\cdot)/|\cdot|^n$. In this article, under the assumption
that the Hardy--Littlewood maximal operator $\mathcal{M}$
is bounded on both $X$ and its associated space,
the authors prove that the commutator $[b,T_\Omega]$
is compact on $X$ if and only if $b\in{\rm CMO}({\mathbb R}^n)$.
To achieve this, the authors mainly employ three key tools:
some elaborate estimates, given in this article,
on the norm in $X$ of the commutators and the characteristic functions of
some measurable subset,
which are implied by the assumed boundedness
of ${\mathcal M}$ on $X$ and its associated space as well as the geometry of
$\mathbb R^n$; the complete John--Nirenberg inequality in $X$
obtained by Y. Sawano et al.; the generalized Fr\'{e}chet--Kolmogorov
theorem on $X$ also established in this article.
All these results have a wide range of applications. Particularly,
even when $X:=L^{p(\cdot)}({\mathbb R}^n)$ (the variable Lebesgue space),
$X:=L^{\vec{p}}({\mathbb R}^n)$ (the mixed-norm Lebesgue space),
$X:=L^\Phi({\mathbb R}^n)$ (the Orlicz space),
and $X:=(E_\Phi^q)_t({\mathbb R}^n)$ (the Orlicz-slice space
or the generalized amalgam space), all these results are new.}
\end{minipage}
\end{center}

\vspace{0.1cm}

\tableofcontents

\vspace{0.1cm}

\section{Introduction\label{s1}}

Let $\Omega$ be a Lipschitz function on the unit sphere of ${\mathbb R}^n$,
which is homogeneous of degree zero and has mean value zero,
namely,
\begin{align}\label{lip}
|\Omega(x)-\Omega(y)|\le|x-y|\ {\rm for\ any\ } x,\ y\in\mathbb{S}^{n-1},
\end{align}
\begin{equation}\label{deg0}
\Omega(\mu x):=\Omega(x)\ {\rm for\ any\ } \mu\in(0,\infty)
\ {\rm and\ }x\in\mathbb{S}^{n-1},
\end{equation}
and
\begin{equation}\label{mean0}
\int_{\mathbb{S}^{n-1}}\Omega(x)\,d\sigma(x)=0,
\end{equation}
here and thereafter,
$\mathbb{S}^{n-1}:=\{x\in\rn:\ |x|=1\}$ denotes the unit sphere in $\rn$
and $d\sigma$ the area measure on $\mathbb{S}^{n-1}$.
To study the factorization theorem of the Hardy space,
Coifman et al. \cite{crw} initiated the study
of commutator $[b,T_\Omega](f):=bT_\Omega(f)-T_\Omega(bf)$, where
$b\in{\rm BMO}(\rn)$ and $T_\Omega$ denotes the Calder\'on--Zygmund operator defined by setting,
for any suitable function $f$ and any $x\in\rn$,
\begin{align}\label{T-def}
T_\Omega(f)(x):=\text{p.\,v.}\,\int_{\rn}\frac{\Omega(x-y)}{|x-y|^n}f(y)\,dy:=
\lim_{\varepsilon\to0^+}\int_{\varepsilon<|x-y|<1/\varepsilon}\frac{\Omega(x-y)}{|x-y|^n}f(y)\,dy,
\end{align}
here and thereafter, $\varepsilon\to0^+$ means $\varepsilon\in(0,\fz)$ and $\varepsilon\to0$.
The commutator of this type plays key roles in harmonic analysis
(see, for instance, \cite{an18,an20,cd15,gos15,ns17,ss08}),
partial differential equations
(see, for instance, \cite{cdd20,cdh16,tyy20}),
and quasiregular mappings
(see, for instance, \cite{i92}).

The first significant result in this direction was made
by Coifman et al. \cite{crw}, which characterizes the boundedness of
such type commutators on the Lebesgue space $L^p(\rn)$ with $p\in(1,\fz)$,
via the well-known space $\BMO(\rn)$.
Recall that the space $\BMO(\rn)$,
introduced by John and Nirenberg \cite{JN},
is defined to be
the set of all locally integrable functions $f$ on $\rn$ such that
\begin{equation*}
\|f\|_{\BMO(\rn)}:=\sup_{\mathrm{ball\,}B\subset\rn}
\frac{1}{|B|}\int_{B}|f(x)-f_B|\,dx<\fz,
\end{equation*}
where the supremum is taken over all balls $B\subset\rn$,
and $f_B:=\frac{1}{|B|}\int_{B}f(x)\,dx$ for any given ball $B\subset\rn$.
Precisely, Coifman et al. \cite{crw} proved that, if a function $b\in {\rm BMO}(\mathbb R^n)$, then the
commutator $[b,T_\Omega]$ is bounded on $L^p(\mathbb R^n)$ for any given $p\in(1, \infty)$,
and also that, if $[b, R_j]$ is bounded on $L^p(\mathbb R^n)$
for any Riesz transform $R_ j$, $j\in\{1,\ldots, n\}$,
then $b\in {\rm BMO}(\mathbb R^n)$.
Moreover, Uchiyama \cite{u} proved that $[b, T_\Omega]$ is bounded on $L^p(\rn)$
for any given $p\in(1,\infty)$ if and only if $b\in\BMO(\rn)$.
Later, such boundedness characterizations were also established
on various function spaces: for instance,
Di Fazio and Ragusa \cite{dr91} on Morrey spaces,
Lu et al. \cite{ldy} on weighted Lebesgue spaces,
and Karlovich and Lerner \cite{KL05} on variable Lebesgue spaces.

As for the compactness characterizations of commutators,
Uchiyama \cite{u} first proved that
$[b, T_\Omega]$ is compact on $L^p(\rn)$ for any given $p\in(1,\infty)$
if and only if $b\in{\rm CMO}(\rn)$,
where ${\rm CMO}(\rn)$ denotes the closure of smooth functions
with compact support in $\BMO(\rn)$.
This characterization of compactness was also extended to Morrey spaces in \cite{CDW12},
and to weighted Lebesgue spaces in \cite{gwy20, CC13}.
However, to the best of our knowledge, for other known function spaces,
such as mixed-norm Lebesgue spaces, variable Lebesgue spaces,
Orlicz spaces, and Orlicz-slice spaces
(see, respectively, Subsections \ref{s5.2}, \ref{s5.3}, \ref{s5.5x},
and \ref{s5.5} below for their histories and definitions),
the equivalent characterization of the compactness of commutators corresponding to
these aforementioned spaces
are still unknown so far. Therefore, it is natural to ask whether or not
there exists a unified theory on
the equivalent characterization for the boundedness and the compactness of commutators
on all aforementioned function spaces.
In this article, we give an affirmative answer to this question
on so-called ball Banach function spaces.

Recall that the ball (quasi-)Banach function space
was introduced by Sawano et al. \cite{SHYY}
(see also Definition \ref{Debqfs} below),
which contains all aforementioned function spaces
as special cases.
For more studies of ball Banach function spaces,
we refer the readers to \cite{s,is,ins,WYYZ,ZYYW}.
Very recently, Chaffee and Cruz-Uribe \cite{CC18},
and Guo et al. \cite{glw} studied the necessity of
the boundedness of commutators on ball Banach function spaces.
However, the sufficiency of the boundedness
of commutators on the ball Banach function space $X$ and
the equivalent characterization of their compactness on $X$
are still unknown.

In what follows, we always let $(X,\|\cdot\|_X)$ be a ball Banach function space
satisfying either or both of the following additional assumptions
(we will explicitly indicate this in the context).
\begin{assumption}\label{assum}
\begin{enumerate}
\item[{\rm(i)}] The Hardy--Littlewood maximal operator $\cm$
[see \eqref{mm} below for its definition]
is bounded on $X$ and $X'$;
here and thereafter, $X'$ denotes the associate space of $X$
(see Definition \ref{def-X'} below for its definition).

\item[{\rm(ii)}] There exists an $s\in(1,\infty)$ such that
$X^{1/s}$ is a ball Banach function space,
where $X^{1/s}$ denotes the $\frac 1s$-convexification of $X$
(see Definition \ref{Debf} below for its definition),
and that $\cm$ is bounded on $(X^{1/s})'$.
\end{enumerate}
\end{assumption}
Motivated by the aforementioned results, in this article,
we establish the following equivalent characterizations of the boundedness and
the compactness of commutators on $X$.
\begin{theorem}\label{thm0}
Let $X$ be a ball Banach function space satisfying Assumption \ref{assum},
$\Omega$ a homogeneous function of degree zero satisfying \eqref{lip}, \eqref{deg0}, and \eqref{mean0},
$T_\Omega$ as in \eqref{T-def}, and $b\in L^1_\loc(\rn)$.
Then
\begin{itemize}
\item[{\rm(i)}]
$[b,T_\Omega]$ is bounded on $X$ if and only if $b\in{\rm BMO}({\mathbb R}^n)$;

\item[{\rm(ii)}]
$[b,T_\Omega]$ is compact on $X$ if and only if $b\in{\rm CMO}({\mathbb R}^n)$.
\end{itemize}
\end{theorem}
Indeed, we prove Theorem \ref{thm0} under much weaker assumptions on $\Omega$;
see Theorems \ref{thm-bdd-1} and \ref{thm-bdd-2}
[and also Remark \ref{guo}(ii)] below for the boundedness,
as well as Theorems \ref{thm-cpt} and \ref{thm-cpt2}
(and also Remark \ref{rem-omega-cpt}) below for the compactness.
To obtain these results, we need to overcome the essential difficulty
caused by the lack of the explicit expression of the norm of $X$,
via mainly employing three key tools:
some elaborate lower and upper estimates,
obtained in Propositions \ref{thmm2} and \ref{thmm3} below,
on the norm in $X$ of the commutators and the characteristic functions of
some measurable subset,
which are implied by the assumed boundedness
of ${\mathcal M}$ on $X$ and its associated space as well as the geometry of
$\mathbb R^n$; the complete John--Nirenberg inequality in $X$
obtained by Sawano et al. in \cite{ins}; the generalized Fr\'{e}chet--Kolmogorov
theorem on $X$ established in Theorem \ref{l-fre} below.
All these results have a wide range of applications,
which not only recover several well-known results
but also yield some new ones. Particularly,
even when $X:=L^{p(\cdot)}({\mathbb R}^n)$ (the variable Lebesgue space),
$X:=L^{\vec{p}}({\mathbb R}^n)$ (the mixed-norm Lebesgue space),
$X:=L^\Phi({\mathbb R}^n)$ (the Orlicz space),
and $X:=(E_\Phi^q)_t({\mathbb R}^n)$ (the Orlicz-slice space
or the generalized amalgam space), all these results are new.
It should be mentioned that, applying the necessity of boundedness,
obtained in Theorem \ref{thm-bdd-2} below,
into six concrete examples of ball Banach function spaces in Section \ref{s5},
we obtain even better results than \cite{CC18} and \cite{glw}
for the necessity of the boundedness of commutators.
In addition, the equivalent characterization of the compactness,
obtained in Theorems \ref{thm-cpt} and \ref{thm-cpt2} below,
coincides with Guo et al. \cite[Theorems 1.4 and 1.5]{gwy20}
about the convolutional singular integral operator on the weighted Lebesgue space.

To be precise, the remainder of this article is organized as follows.

In Section \ref{s2},
we first show that $[b,T_\Omega]$ is bounded on $X$
for any given $b\in{\rm BMO}(\rn)$ in Theorem \ref{thm-bdd-1}
via the extrapolation theorem.
It should be pointed out that ball Banach function spaces
are embedded into weighted Lebesgue spaces (see Lemma \ref{embed1} below),
which guarantees that the considered Calder\'on--Zygmund commutator is well defined on
ball Banach function spaces (see Proposition \ref{thm-bdd-0} below).
Observe that the extrapolation theorem plays an essential role
in establishing the boundedness of operators on ball Banach function spaces,
which is a bridge connecting the ball Banach function space and
the weighted Lebesgue space.
Moreover, combining the technique of the local mean oscillation as in \cite{lor,gwy20}
and a fine inequality on the norm in $X$ (see Lemma \ref{rh2} below),
we also show that, if $[b,T_\Omega]$ is bounded on $X$,
then $b\in\BMO(\rn)$ in Theorem \ref{thm-bdd-2} below.
As a consequence, Theorem \ref{thm0}(ii) is a direct corollary of
Theorems \ref{thm-bdd-1} and \ref{thm-bdd-2}.

Section \ref{s3} is devoted to Theorem \ref{thm0}(ii)
which can be easily deduced from two more general results:
Theorem \ref{thm-cpt} below (the sufficiency)
and Theorem \ref{thm-cpt2} below (the necessity).
To prove these, we need to overcome some essential difficulties
by borrowing some basic ideas from the proof of
the recent result on the weighted Lebesgue space given by Guo et al. \cite{gwy20};
see also Uchiyama \cite{u} for the corresponding one on the Lebesgue space,
and Chen et al. \cite{CDW12} for the corresponding one on the Morrey space.
However, their calculations are no longer completely feasible for the ball
Banach function space $X$ because they need to use the following three crucial properties of
the considered norm, which are not available for $\|\cdot\|_X$:
the Lebesgue dominated convergence theorem,
the translation invariance, and the explicit expression of the norm.
In the proof of Theorem \ref{thm-cpt},
using a skillful decomposition and the
smooth truncated technique given by Clop and Cruz \cite{CC13},
we avoid the translation invariance as in Uchiyama \cite{u}
or Chen et al. \cite{CDW12}. Moreover, we establish a new criterion
on the compactness of a set in the ball Banach function space $X$ (see Theorem \ref{l-fre} below),
which generalizes the Fr\'{e}chet--Kolmogorov theorem
in \cite{CDW12,CC13,gz20} to the present setting. Indeed, via establishing
a new Minkowski-type inequality
for ball quasi-Banach function spaces $X$
(see Lemma \ref{LeMI} below),
we drop the assumption that $X$ has a absolutely continuous norm
in \cite[Theorem 3.1]{gz20}.
In the proof of Theorem \ref{thm-cpt2},
since we do not have the aforementioned three key properties on the norm $\|\cdot\|_X$,
nearly all the corresponding calculations used in \cite{u,CDW12,gwy20} are unworkable
in the present setting. To overcome these difficulties, we need to improve the method used in \cite{gwy20}.
Indeed, we first establish the lower estimates of commutators in Proposition \ref{thmm2}
via the aforementioned technique of the local mean oscillation;
we then apply an equivalent characterization of BMO$(\rn)$
via the ball Banach function space obtained in \cite{ins} (see Lemma \ref{BMO-X} below)
to establish the upper estimates of commutators (see Proposition \ref{thmm3} below);
from Propositions \ref{thmm2} and \ref{thmm3}, we finally deduce
the desired necessity of the equivalent characterization on the compactness
of commutators.

In Section \ref{s5}, we apply all these results obtained in Sections \ref{s2} and \ref{s3},
respectively, to $X:=M_r^p({\mathbb R}^n)$ (the Morrey space) or to
$X:=L_\omega^p({\mathbb R}^n)$ (the weighted Lebesgue space), and we find
that, even for these well-known function spaces, some of our results also improve the known results
(see Remark \ref{rem-bdd-Morrey} below for more details).
Moreover, to the best of our knowledge,
when we apply all these results obtained in Sections \ref{s2} and \ref{s3},
respectively, to $X:=L^{p(\cdot)}({\mathbb R}^n)$ (the variable Lebesgue space),
$X:=L^{\vec{p}}(\rn)$ (the mixed-norm Lebesgue space),
$X:=L^\Phi({\mathbb R}^n)$ (the Orlicz space),
or $X:=(E_\Phi^r)_t({\mathbb R}^n)$ (the Orlicz-slice space or the generalized amalgam space),
all these results are totally new.

Finally, we make some conventions on notation.
Let $\nn:=\{1,2,\ldots\}$, $\zz_+:=\nn\cup\{0\}$,
and $\zz_+^n:=(\zz_+)^n$. We always denote by $C$ a \emph{positive constant}
which is independent of the main parameters, but it may vary from line to line.
We also use $C_{(\alpha,\beta,\ldots)}$ to denote a positive constant depending
on the indicated parameters $\alpha,\beta,\ldots.$ The \emph{symbol} $f\lesssim g$ means that $f\le Cg$.
If $f\lesssim g$ and $g\lesssim f$, we then write $f\sim g$.
If $f\le Cg$ and $g=h$ or $g\le h$, we then write $f\ls g\sim h$
or $f\ls g\ls h$, \emph{rather than} $f\ls g=h$
or $f\ls g\le h$. The \emph{symbol} $\lfloor s\rfloor$  for any $s\in\mathbb{R}$
denotes the largest integer not greater
than $s$. We use $\vec0_n$ to denote the \emph{origin} of $\rn$ and let
$\mathbb{R}^{n+1}_+:=\rn\times(0,\infty)$.
If $E$ is a subset of $\rn$, we denote by $\mathbf{1}_E$ its
characteristic function and by $E^\complement$ the set $\rn\setminus E$.
Furthermore,
for any $\alpha\in(0,\infty)$ and any ball $B:=B(x_B,r_B)$ in $\rn$, with $x_B\in\rn$ and
$r_B\in(0,\infty)$, we let $\alpha B:=B(x_B,\alpha r_B)$.
Finally, for any $q\in[1,\infty]$, we denote by $q'$ its \emph{conjugate exponent},
namely, $1/q+1/q'=1$.

\section{Boundedness characterization of commutators on ball Banach\\
function spaces\label{s2}}

In this section, we first present some known facts
on the ball quasi-Banach function space $X$ in Subsection \ref{s2.1},
and then establish the characterization of
the boundedness of commutators in Subsection \ref{s2.2}.

\subsection{Ball quasi-Banach function spaces\label{s2.1}}

We now recall some preliminaries on ball quasi-Banach function spaces
introduced in \cite{SHYY}.
Denote by the \emph{symbol} $\mathscr M(\rn)$ the set of
all measurable functions on $\rn$.
For any $x\in\rn$ and $r\in(0,\infty)$, let $B(x,r):=\{y\in\rn:\ |x-y|<r\}$ and
\begin{equation}\label{Eqball}
\BB:=\lf\{B(x,r):\ x\in\rn\quad\text{and}\quad r\in(0,\infty)\r\}.
\end{equation}

\begin{definition}\label{Debqfs}
A quasi-Banach space $X\subset\mathscr M(\rn)$ is called a \emph{ball quasi-Banach function space} if it satisfies
\begin{itemize}
\item[(i)] $\|f\|_X=0$ implies that $f=0$ almost everywhere;
\item[(ii)] $|g|\le |f|$ almost everywhere implies that $\|g\|_X\le\|f\|_X$;
\item[(iii)] $0\le f_m\uparrow f$ almost everywhere implies that $\|f_m\|_X\uparrow\|f\|_X$;
\item[(iv)] $B\in\BB$ implies that $\mathbf{1}_B\in X$, where $\BB$ is as in \eqref{Eqball}.
\end{itemize}
Moreover, a ball quasi-Banach function space $X$ is called a
\emph{ball Banach function space} if the norm of $X$
satisfies the triangle inequality: for any $f,\ g\in X$,
\begin{equation}\label{eq22x}
\|f+g\|_X\le \|f\|_X+\|g\|_X,
\end{equation}
and, for any $B\in \BB$, there exists a positive constant $C_{(B)}$, depending on $B$, such that, for any $f\in X$,
\begin{equation*}\label{eq2.3}
\int_B|f(x)|\,dx\le C_{(B)}\|f\|_X.
\end{equation*}
\end{definition}

\begin{remark}\label{ball-bounded}
\begin{itemize}
\item[{\rm(i)}]
Observe that, in Definition \ref{Debqfs}, if we replace any ball $B$ by
any \emph{bounded} measurable set $E$, we obtain its another equivalent formulation.

\item[{\rm(ii)}]
Recall that a quasi-Banach space $X\subset\mathscr M(\rn)$ is called a \emph{quasi-Banach function space} if it
is a ball quasi-Banach function space and it satisfies Definition \ref{Debqfs}(iv) with ball
replaced by any measurable set of \emph{finite measure} (see, for instance, \cite[Chapter 1, Definitions 1.1 and 1.3]{BS}).
It is easy to see that every quasi-Banach function space is a ball quasi-Banach function space, and the converse
is not necessary to be true.
As was mentioned in \cite[p.\,9]{SHYY} and \cite[Section 5]{WYYZ}, the family of ball Banach function spaces includes Morrey spaces, mixed-norm Lebesgue spaces,  variable Lebesgue spaces,  weighted Lebesgue spaces, and Orlicz-slice spaces,
which are not necessary to be Banach function spaces.
\end{itemize}
\end{remark}

The following notion of the associate space of a ball Banach function space
can be found, for instance, in \cite[Chapter 1, Definitions 2.1 and 2.3]{BS}.

\begin{definition}\label{def-X'}
For any ball Banach function space $X$, the \emph{associate space} (also called the
\emph{K\"othe dual}) $X'$ is defined by setting
\begin{equation}\label{asso}
X':=\lf\{f\in\mathscr M(\rn):\ \|f\|_{X'}:=
\sup_{\{g\in X:\ \|g\|_X=1\}}\|fg\|_{L^1(\rn)}<\infty\r\},
\end{equation}
where $\|\cdot\|_{X'}$ is called the \emph{associate norm} of $\|\cdot\|_X$.
\end{definition}

\begin{remark}\label{bbf}
By \cite[Proposition 2.3]{SHYY}, we know that, if $X$ is a ball Banach function space,
then its associate space $X'$ is also a ball Banach function space.
\end{remark}

The following lemma is just \cite[Lemma 2.6]{ZWYY}.
\begin{lemma}\label{Lesdual}
Let $X$ be a ball quasi-Banach function space
satisfying the triangle inequality as in \eqref{eq22x}.
Then $X$ coincides with its second associate space $X''$.
In other words, a function $f$ belongs to $X$ if and only if it belongs to $X''$ and,
in that case,
$$
\|f\|_X=\|f\|_{X''}.
$$
\end{lemma}

The following H\"older inequality is a direct corollary of both Definition \ref{Debqfs}(i)
and \eqref{asso} (see \cite[Theorem 2.4]{BS}).

\begin{lemma}\label{LeHolder}
Let $X$ be a ball quasi-Banach function space
satisfying the triangle inequality as in \eqref{eq22x},
and $X'$ its associate space.
If $f\in X$ and $g\in X'$, then $fg$ is integrable and
\begin{equation*}\label{12}
\int_\rn|f(x)g(x)|\,dx\le \|f\|_X\|g\|_{X'}.
\end{equation*}
\end{lemma}

We still need to recall the notion of the convexity of ball quasi-Banach spaces,
which is a part of \cite[Definition 2.6]{SHYY}.
\begin{definition}\label{Debf}
Let $X$ be a ball quasi-Banach function space and $p\in(0,\infty)$.
The $p$-\emph{convexification} $X^p$ of $X$ is defined by setting $X^p:=\{f\in\mathscr M(\rn):\ |f|^p\in X\}$
equipped with the quasi-norm $\|f\|_{X^p}:=\||f|^p\|_X^{1/p}$.
\end{definition}

In what follows, we denote by the \emph{symbol $L_{\loc}^1(\rn)$}
the set of all locally integrable functions on $\rn$.
The \emph{Hardy--Littlewood maximal operator} $\cm$
is defined by setting, for any $f\in L_{\loc}^1(\rn)$ and $x\in\rn$,
\begin{equation}\label{mm}
\cm(f)(x):=\sup_{B\ni x}\frac1{|B|}\int_B|f(y)|\,dy,
\end{equation}
where the supremum is taken over all balls $B\in\BB$ containing $x$.

For any $\theta\in(0,\infty)$, the \emph{powered Hardy--Littlewood
maximal operator} $\cm^{(\theta)}$ is defined by setting,
for any $f\in L_{\loc}^1(\rn)$ and $x\in\rn$,
\begin{equation}\label{mmx}
\cm^{(\theta)}(f)(x):=\lf\{\cm\lf(|f|^\theta\r)(x)\r\}^{1/\theta}.
\end{equation}

The following lemma is a part of \cite[Remark 2.19(i)]{ZWYY}.

\begin{lemma}\label{as1}
Let $\theta\in(0,\infty)$ and
$X$ be a ball quasi-Banach function space. Assume that there exists a positive
constant $C$ such that, for any $f\in\mathscr M(\rn)$,
\begin{equation*}\label{as1-1}
\lf\|\cm^{(\theta)}(f)\r\|_X\le C\lf\|f\r\|_X.
\end{equation*}
Then there exists a positive
constant $\widetilde{C}$ such that, for any ball $B\in\BB$
and $\beta\in[1,\infty)$,
\begin{align}\label{EqHLMS}
\lf\|\mathbf{1}_{\beta B}\r\|_{X}
\le \widetilde{C}\beta^{n/\theta}\lf\|\mathbf{1}_{B}\r\|_{X},
\end{align}
where the positive constant $C$ is independent of $B\in\BB$ and $\beta$.
\end{lemma}

\begin{remark}\label{eta}
From \cite[Lemma 2.15(ii)]{SHYY}, we deduce that, if $\cm$ is bounded on $X$,
then there exists an $\eta\in(1,\infty)$ such that $\cm^{(\eta)}$ is bounded on $X$,
where $\cm^{(\eta)}$ is as in \eqref{mmx} with $\tz$ replaced by $\eta$.
\end{remark}

\subsection{Sufficiency and necessity of boundedness of commutators\label{s2.2}}

In this subsection, we obtain the sufficiency and the necessity of the boundedness
of commutators, respectively, in Theorems \ref{thm-bdd-1} and \ref{thm-bdd-2} below.

First, we recall the notions of Muckenhoupt weights $A_p(\rn)$ (see, for instance, \cite{G1}).
\begin{definition}\label{weight}
An \emph{$A_p(\rn)$-weight} $\omega$, with $p\in[1,\infty)$, is a
locally integrable and nonnegative function on $\rn$ satisfying that,
when $p\in(1,\infty)$,
\begin{equation*}
[\oz]_{A_p(\rn)}:=\sup_{B\in\BB}\lf[\frac1{|B|}\int_B\omega(x)\,dx\r]\lf\{\frac1{|B|}
\int_B\lf[\omega(x)\r]^{\frac1{1-p}}\,dx\r\}^{p-1}<\infty,
\end{equation*}
and, when $p=1$,
$$
[\omega]_{A_1(\rn)}:=\sup_{B\in\BB}\frac1{|B|}\int_B\omega(x)\,dx\lf[\lf\|\omega^{-1}\r\|_{L^\infty(B)}\r]<\infty,
$$
where $\BB$ is as in \eqref{Eqball}. Define $A_\infty(\rn):=\bigcup_{p\in[1,\infty)}A_p(\rn)$.
\end{definition}

\begin{definition}\label{wk}
Let $p\in(0,\infty)$ and $\omega\in A_\infty(\rn)$.
The \emph{weighted Lebesgue space $L_\omega^p(\rn)$} is defined
to be the set of all measurable functions $f$ on $\rn$ such that
$$
\|f\|_{L^p_\omega(\rn)}:=\lf[\int_\rn|f(x)|^p\omega(x)\,dx\r]^\frac1p<\infty.
$$
\end{definition}

The following technical lemma is just \cite[Lemma 4.7]{cwyz}, which
plays a vital role in the proof of Proposition \ref{thm-bdd-0} below.

\begin{lemma}\label{embed1}
Let $X$ be a ball quasi-Banach function space satisfying Assumption \ref{assum}(ii).
Then there exists an $\epsilon\in(0,1)$ such that
$X$
continuously embeds into $L_\omega^s(\rn)$ with $\omega:=[\cm(\mathbf1_{B(\vec0_n,1)})]^\epsilon\in A_1(\rn)$,
namely, there exists a positive constant $C$ such that, for any $f\in X$,
$$\|f\|_{L_\omega^s(\rn)}\le C\|f\|_X.$$
\end{lemma}

The following extrapolation theorem is just \cite[Lemma 7.34]{ZWYY},
which is a slight variant of a special case of \cite[Theorem 4.6]{CMP} via
replacing Banach function spaces by ball Banach function spaces.

\begin{lemma}\label{thet}
Let $X$ be a ball quasi-Banach function space and $p_0\in(0,\infty)$.
Let $\mathcal{F}$ be the set of all
pairs of nonnegative measurable
functions $(F,G)$ such that, for any given $\omega\in A_1(\rn)$,
$$
\int_{\rn}[F(x)]^{p_0}\omega(x)\,dx\leq C_{(p_0,[\omega]_{A_1(\rn)})}
\int_{\rn}[G(x)]^{p_0}\omega(x)\,dx,
$$
where $C_{(p_0,[\omega]_{A_1(\rn)})}$ is a positive constant independent
of $(F,G)$, but depends on $p_0$ and $[\omega]_{A_1(\rn)}$.
Assume that there exists a $q_0\in[p_0,\infty)$ such that $X^{1/q_0}$ is
a ball Banach function space and $\cm$ is bounded on $(X^{1/q_0})'$,
where $\cm$ is as in \eqref{mm}. Then there
exists a positive constant $C_0$ such that, for any $(F,G)\in\mathcal{F}$,
$$
\|F\|_{X}\leq C_0\|G\|_{X}.
$$
\end{lemma}

To study the boundedness of commutators in this article,
we modify Lemma \ref{thet} as follows.

\begin{proposition}\label{thm-bdd-0}
Let $X$ be a ball quasi-Banach function space satisfying Assumption \ref{assum}(ii).
Let $\mathcal{T}$ be an operator satisfying,
for any given $\oz\in A_1(\rn)$ and any $f\in L^s_\oz(\rn)$ ,
$$\|\mathcal{T} (f)\|_{L^s_\oz(\rn)}\le
C_{(s,[\oz]_{A_1(\rn)})}\|f\|_{L^s_\oz(\rn)},$$
where $C_{(s,[\oz]_{A_1(\rn)})}$ is a positive constant independent of $f$,
but depends on $s$ and $[\oz]_{A_1(\rn)}$.
Then there
exists a positive constant $C$ such that, for any $f\in X$,
\begin{equation}\label{Eqcc0}
\lf\|\mathcal{T} (f)\r\|_X\le C\lf\|f\r\|_X.
\end{equation}
\end{proposition}

\begin{proof}
Let $X$ be a ball quasi-Banach function space and $s\in(1,\infty)$.
Assume that $X^{1/s}$ is a ball Banach function space
and $\cm$ is bounded on $(X^{1/s})'$.
To show \eqref{Eqcc0}, let
$$
\cf:=\left\{(|\mathcal{T}(f)|, |f|):\
f\in\bigcup_{\oz\in A_1(\rn)} L^s_\oz(\rn)\right\}.
$$
Then, by the assumption on $\mathcal{T}$,
we obtain, for any given $\omega\in A_1(\rn)$ and
any $f\in\bigcup_{\oz\in A_1(\rn)} L^s_\oz(\rn)$,
$$
\int_{\rn}\lf|\mathcal{T}(f)(x)\r|^s\omega(x)\,dx
\le \lf[C_{(s,[\oz]_{A_1(\rn)})}\r]^s\int_{\rn}\lf|f(x)\r|^s\omega(x)\,dx,
$$
which, together with the assumptions that $X^{1/s}$ is a ball Banach function space and
$\cm$ is bounded on $(X^{1/s})'$, and Lemma \ref{thet}, further implies that,
for any $f\in\bigcup_{\oz\in A_1(\rn)} L^s_\oz(\rn)$,
\begin{equation}\label{eq3.5.0}
\lf\|\mathcal{T}(f)\r\|_{X}\lesssim \lf\|f\r\|_{X}.
\end{equation}
By Lemma \ref{embed1}, we know that $X\subset \bigcup_{\oz\in A_1(\rn)} L^s_\oz(\rn)$,
which, combined with \eqref{eq3.5.0}, implies the desired boundedness and hence
completes the proof of Proposition \ref{thm-bdd-0}.
\end{proof}

In order to introduce singular integral operators with homogeneous kernel,
we now state the following notion of the $L^\infty$-Dini condition.
\begin{definition}\label{a2.15}
A function $\Omega\in L^\infty(\mathbb{S}^{n-1})$
is said to satisfy the \emph{$L^\infty$-Dini condition}
if
\begin{align}\label{Dini}
\int_{0}^1
\frac{\omega_\infty(\tau)}{\tau}\,d\tau<\infty,
\end{align}
where, for any $\tau\in(0,1)$,
$$
\omega_\infty(\tau):=
\sup_{\{x,\ y\in\mathbb{S}^{n-1}:\ |x-y|<\tau\}}
|\Omega(x)-\Omega(y)|.
$$
\end{definition}
Recall that the \emph{symbol} $a\to0^+$ means that $a\in(0,\fz)$ and $a\to0$.
Through this article,
assuming that $\Omega$ satisfies \eqref{deg0}, \eqref{mean0},
and the $L^\infty$-Dini condition,
a linear operator $T_\Omega$ is called
a \emph{singular integral operator with homogeneous kernel} $\Omega$
(see, for instance, \cite[p.\,53, Corollary 2.1.1]{ldy})
if, for any $f\in L^p(\rn)$ with $p\in[1,\infty)$, and for any
$x\in\rn$, \eqref{T-def} holds true.
Let $X$ be a ball quasi-Banach function space.
Assume that there exists an $s\in(0,\infty)$ such that $X^{1/s}$ is a ball Banach function space
and $\cm$ is bounded on $(X^{1/s})'$, where $\cm$ is as in \eqref{mm}.
By Lemma \ref{embed1} and \cite[Corollary 7.13]{d}, we know that,
for any $f\in X$, $T_\Omega(f)(x)$ exists for almost every $x\in\rn$.

For any given $b\in L_{\loc}^1(\rn)$,
the commutator $[b, T_\Omega]$ is defined by setting,
for any bounded function $f$ with compact support, and for any $x\in\rn$,
\begin{equation}\label{z}
[b,T_\Omega](f)(x):=b(x)T_\Omega(f)(x)-T_\Omega(bf)(x).
\end{equation}

To prove Theorem \ref{thm-bdd-1}, we need the following weighted
$L^p(\rn)$ boundedness of the commutator $[b,T_\Omega]$, which is a part of
\cite[Theorem 2.4.4]{ldy}.

\begin{lemma}\label{coa}
Let $\omega\in A_1(\rn)$, $p\in(1,\infty)$, and $q\in(1,\infty]$ satisfy $q'\le p$ with $1/q+1/q'=1$.
Assume that $b\in\BMO(\rn)$,
$\Omega\in L^q(\mathbb{S}^{n-1})$ satisfies \eqref{deg0} and \eqref{mean0}, and
$T_\Omega$ is
a singular integral operator with homogeneous kernel $\Omega$.
Then there exists
a positive constant $C_{(p,\Omega,[\oz]_{A_p(\rn)})}$,
depending on $p$, $\Omega$, and $[\oz]_{A_p(\rn)}$, such that,
for any $f\in L^p_\oz(\rn)$,
$$
\int_\rn\lf|[b,T_\Omega](f)(x)\r|^p\omega(x)\,dx\leq
C_{(p,\Omega,[\oz]_{A_p(\rn)})}\|b\|_{\BMO(\rn)}^p\int_\rn\lf|f(x)\r|^p\omega(x)\,dx.
$$
\end{lemma}

Then we immediately have the following sufficiency of the boundedness of
commutators on ball Banach function spaces.

\begin{theorem}\label{thm-bdd-1}
Let $X$ be a ball Banach function space satisfying Assumption \ref{assum}(ii)
for some given $s\in(1,\fz)$.
Let $q\in(1,\infty]$ satisfy $q'\le s$ with $1/q+1/q'=1$.
Assume that $b\in\BMO(\rn)$,
$\Omega\in L^q(\mathbb{S}^{n-1})$ satisfies \eqref{deg0} and \eqref{mean0}, and
$T_\Omega$ is
a singular integral operator with homogeneous kernel $\Omega$.
Then there exists a positive constant $C$ such that,
for any $f\in X$,
\begin{equation*}\label{Eqcc}
\lf\|[b,T_\Omega](f)\r\|_X\le C\|b\|_{\BMO(\rn)}\lf\|f\r\|_X.
\end{equation*}
\end{theorem}

\begin{proof}
Using Lemma \ref{coa} and Proposition \ref{thm-bdd-0}, we immediately complete the proof of Theorem \ref{thm-bdd-1}.
\end{proof}

\begin{remark}
Let $X:=L^p(\rn)$ with $p\in(1,\infty)$.
Assume that
$\Omega\in L^\infty(\mathbb{S}^{n-1})$ satisfies \eqref{lip}, \eqref{deg0}, and \eqref{mean0}.
Then, in this case,
Theorem \ref{thm-bdd-1} coincides with the classical conclusion in \cite[Theorem 1]{crw}.
Compared with the assumptions on $\Omega$ in \cite[Theorem 1]{crw},
the assumptions on $\Omega$ in Theorem \ref{thm-bdd-1} are much weaker.
\end{remark}

Now, we show the necessity of the boundedness of commutators.
To this end, we need three key lemmas,
namely, Lemmas \ref{GR}, \ref{thmb3}, and \ref{rh2},
respectively.

First, recall that, for any given measurable function $f$,
the \emph{non-increasing rearrangement} of $f$
is defined by setting, for any $t\in(0,\infty)$,
$$f^*(t):=\inf\{\alpha\in(0,\fz):\ |\{x\in\rn:\ |f(x)|>\alpha\}|<t\};$$
for any given $f\in L_{\loc}^1(\rn)$ and ball $B\subset\rn$,
the \emph{local mean oscillation} of $f$ on $B$ is defined by setting, for any $\lambda\in(0,1)$,
\begin{equation}\label{z3}
\oz_{\lz}(f;B):=
\inf_{c\in\cc}\lf\{\lf[(f-c)\mathbf{1}_{B}\r]^\ast(\lambda|B|)\r\}.
\end{equation}
The following characterization of ${\rm BMO}(\rn)$ is a part of
\cite[Lemma 2.5]{gwy20}; see also \cite[Lemma 2.1]{lor}.

\begin{lemma}\label{GR}
Let $\lambda\in(0,1/2]$. Then there exist a positive constant $C$ such that,
for any $f\in\BMO(\rn)$,
$$C^{-1}\|f\|_{\BMO(\rn)}
\le \sup_{{\rm ball\ }B\subset \rn} \oz_{\lz}(f;B)
\le C \|f\|_{\BMO(\rn)}.$$
\end{lemma}
Moreover, the following geometrical lemma is just
\cite[Proposition 3.1]{gwy20} with cubes replaced by balls.

\begin{lemma}\label{thmb3}
Let $\lambda\in(0,1)$ and $b\in L_{\loc}^1(\rn)$.
Let
$\Omega\in L^\infty(\mathbb{S}^{n-1})$ satisfy \eqref{deg0}, \eqref{mean0}, and
there exists an open set $\Lambda\subset \mathbb{S}^{n-1}$ such that
$\Omega$ does not change sign on $\Lambda$.
Then there exist an $\varepsilon_0\in(0,\infty)$
and a $k_0\in(10\sqrt n,\fz)$, depending only on $\Omega$ and $n$,
such that, for any given ball $B(x_0,r_0)\subset\rn$
with $x_0\in\rn$ and $r_0\in(0,\infty)$,
there exist an $x_1\in\rn$ and
 measurable sets $E\subset B(x_0,r_0)$ with $|E|=\frac{\lambda}{2}|B(x_0,r_0)|$,
$F\subset B(x_1,r_0)$ with $|x_1-x_0|=2k_0r_0$ and $|F|=\frac{1}{2}|B(x_1,r_0)|$,
and $G\subset E\times F$ with
$|G|\geq\frac{\lambda}{8} |B(x_0,r_0)|^2$ satisfying the following properties:
\begin{itemize}
\item[{\rm(i)}]for any $x\in E$ and
$y\in F$, $\oz_{\lz}(b;B)\le|b(x)-b(y)|$;
\item[{\rm(ii)}]$\Omega(\frac{x-y}{|x-y|})$ and $b(x)-b(y)$ do not change sign on $E\times F$;
\item[{\rm(iii)}]for any $(x,y)\in G$, $|\Omega(\frac{x-y}{|x-y|})|\geq \varepsilon_0$.
\end{itemize}
\end{lemma}

In addition, the following lemma shows that, for any ball $B$,
the reverse of Lemma \ref{LeHolder} also holds true
with $f$ and $g$ replaced by $\mathbf1_{B}$,
which is a part of \cite[Lemma 2.2 and Remark 2.3]{is}.

\begin{lemma}\label{rh2}
Let X be a ball Banach function space such that $\cm$ is bounded on $X$.
Then there exists a positive constant $C$ such that,
for any ball $B\subset\rn$,
$$\frac{1}{|B|}\|\mathbf1_{B}\|_X\|\mathbf1_{B}\|_{X'}\le C.$$
\end{lemma}

In what follows, for any operator $\mathcal{L}$ mapping $X$ into itself,
we use $\|\mathcal{L}\|_{X\to X}$ to denote its operator norm.
Also, it should be pointed out that,
for any given $\Omega\in L^\fz(\mathbb{S}^{n-1})$ and $b\in L^1_\loc(\rn)$,
and any bounded measurable set $F\subset\rn$,
$[b,T_\Omega](\mathbf{1}_F)$ has the following integral representation: for any $x\in\rn\setminus \overline{F}$,
$$[b,T_\Omega](\mathbf{1}_F)(x)
=\int_F [b(x)-b(y)]\frac{\Omega(x-y)}{|x-y|^n}\,dy.$$
\begin{theorem}\label{thm-bdd-2}
Let $X$ be a ball Banach function space and $b\in L_{\loc}^1(\rn)$.
Assume that $\cm$ is bounded on $X$.
Let $\Omega\in L^\infty(\mathbb{S}^{n-1})$ satisfy \eqref{deg0}, \eqref{mean0},
and there exists an open set $\Lambda\subset \mathbb{S}^{n-1}$ such that
$\Omega$ does not change sign on $\Lambda$.
If $[b,T_\Omega]$ is bounded on $X$,
then $b\in {\rm BMO}(\rn)$ and there exist a positive constant $C$,
independent of $b$, such that
$$\|b\|_{{\rm BMO}(\rn)}\le C \lf\|[b,T_\Omega]\r\|_{X\to X}.$$
\end{theorem}
\begin{proof}
Let $\lz\in(0,1/2]$.
To prove this theorem, by Lemma \ref{GR}, it suffices to show that
there exists a positive constant $C$, independent of $b$, such that,
for any ball $B\subset\rn$,
\begin{equation}\label{jj}
\oz_{\lz}(b;B)\le C.
\end{equation}
Let $b\in L_{\loc}^1(\rn)$ and $B:=B(x_0,r_0)$ with $x_0\in\rn$ and $r_0\in(0,\infty)$.
Let $\varepsilon_0$, $k_0$, $G$, $E$, and $F$
be as in Lemma \ref{thmb3}.
Then, by (i) and (iii) of Lemma \ref{thmb3},
we conclude that
$$
\oz_{\lz}(b;B)|G|\le
\frac{1}{\varepsilon_0}\int_{G}|b(x)-b(y)|
\lf|\Omega(x-y)\r|\,dx\,dy.
$$
From this, the fact that $|x-y|\le2(k_0+1)r_0$
for any $(x,y)\in G$, Lemma \ref{thmb3}(ii),
$|G|\geq\frac{\lambda}{8}|B|^2$,
the definition of $[b,T_\Omega]$,
and the observation $\overline{E}\cap \overline{F}=\emptyset$,
we deduce that
\begin{align*}
\oz_{\lz}(b;B)&\le
\frac{[2(k_0+1)r_0]^n}{\varepsilon_0|G|}\int_{G}|b(x)-b(y)|
\frac{|\Omega(x-y)|}{|x-y|^n}\,dx\,dy\\
&\le
\frac{8[2(k_0+1)r_0]^n}{\varepsilon_0\lz|B|^2}
\int_{E}\lf|\int_F [b(x)-b(y)]
\frac{\Omega(x-y)}{|x-y|^n}\,dy\r|\, dx\\
&\le
\frac{8[2(k_0+1)r_0]^n}{\varepsilon_0\lz|B|^2}
\int_{E}\lf|[b,T_\Omega](\mathbf{1}_F)(x)\r|\, dx,
\end{align*}
which, combined with Lemmas \ref{LeHolder} and \ref{rh2}, \eqref{EqHLMS},
and the fact that
$[b,T_\Omega]$ is bounded on $X$,
further implies that
\begin{align*}
\oz_{\lz}(b;B)&\le
\frac{8[2(k_0+1)r_0]^n}{\varepsilon_0\lz|B|^2}
\lf\|[b,T_\Omega](\mathbf{1}_F)\r\|_{X}
\|\mathbf{1}_B\|_{X'}
\ls \frac{1}{|B|} \lf\|[b,T_\Omega]\r\|_{X\to X}
\lf\|\mathbf{1}_{k_0B}\r\|_{X}
\|\mathbf{1}_B\|_{X'}\\
&\ls \lf\|[b,T_\Omega]\r\|_{X\to X}
\frac{1}{|B|}
\lf\|\mathbf{1}_{B}\r\|_{X}
\|\mathbf{1}_B\|_{X'}
\ls \lf\|[b,T_\Omega]\r\|_{X\to X},
\end{align*}
where the implicit positive constants depend only on $\lz$, $k_0$, $\varepsilon_0$, and $n$.
This finishes the proof of \eqref{jj}
and hence of Theorem \ref{thm-bdd-2}.
\end{proof}

\begin{remark}\label{guo}
\begin{itemize}
\item[{\rm(i)}]
The necessity of the boundedness of commutators was also obtained by
Guo et al. \cite[Theorem 2.1]{glw}
under the assumption $\Omega\in C(\mathbb{S}^{n-1})$
and some other specific conditions on the norm of $\Omega$ in ball Banach function spaces.
Observe that Theorem \ref{thm-bdd-2} and \cite[Theorem 2.1]{glw}
can not cover each other.

\item[{\rm(ii)}]
It is easy to see that Theorem \ref{thm0}(i)
is a direct corollary of Theorems \ref{thm-bdd-1} and \ref{thm-bdd-2}.
\end{itemize}
\end{remark}

\section{Compactness characterization of commutators on ball Banach \\ function spaces}\label{s3}

In this section,
applying Theorems \ref{thm-bdd-1} and \ref{thm-bdd-2},
we further investigate the compactness
of the commutator on ball Banach function spaces.

In what follows, the space ${\rm CMO}(\mathbb R^n)$ is defined to be the
closure in $\BMO(\rn)$ of $C^\fz_{\rm c}(\rn)$ [the set of all infinitely
differentiable functions on $\rn$ with compact support].
Recall that the Hardy--Littlewood operator $\cm$ is defined in \eqref{mm},
and the commutator $[b,T_\Omega]$ in \eqref{z}.

\begin{theorem}\label{thm-cpt}
Let $X$ be a ball Banach function space satisfying Assumption \ref{assum}(ii),
$\Omega\in L^\infty(\mathbb{S}^{n-1})$ satisfy \eqref{deg0}, \eqref{mean0}, and \eqref{Dini},
and $T_\Omega$ be a singular integral operator with homogeneous kernel $\Omega$.
Assume that $\cm$ is bounded on $X$.
If $b\in {\rm CMO}(\rn)$, then the commutator $[b,T_\Omega]$ is compact on $X$.
\end{theorem}

\begin{theorem}\label{thm-cpt2}
Let $X$ be a ball Banach function space satisfying Assumption \ref{assum}(i),
and $b\in L_{\loc}^1(\rn)$.
Let
$\Omega\in L^\infty(\mathbb{S}^{n-1})$ satisfy \eqref{deg0} and \eqref{mean0}, and
there exists an open set $\Lambda\subset \mathbb{S}^{n-1}$ such that
$\Omega$ does not change sign on $\Lambda$.
If $[b,T_\Omega]$ is compact on $X$, then $b\in {\rm CMO}(\rn)$.
\end{theorem}

\begin{remark}\label{rem-omega-cpt}
It is easy to see that the assumptions on $\Omega$ in
Theorems \ref{thm-cpt} and \ref{thm-cpt2} are much weaker than
the Lipschitz condition which was also
used in Uchiyama \cite[Theorems 1 and 2]{u},
and hence Theorem \ref{thm0}(ii) is a direct corollary of
Theorems \ref{thm-cpt} and \ref{thm-cpt2}.
\end{remark}

The proofs of Theorems \ref{thm-cpt} and \ref{thm-cpt2} are given, respectively, in
Subsections \ref{s3.1} and \ref{s3.2} below.

\subsection{Proof of Theorem \ref{thm-cpt}}\label{s3.1}

To show Theorem \ref{thm-cpt}, we need several key lemmas.
The first one is the following Minkowski-type inequality
for ball quasi-Banach function spaces.

\begin{lemma}\label{LeMI}
Let $X$ be a ball quasi-Banach function space
satisfying the triangle inequality as in \eqref{eq22x},
$E$ a measurable subset of $\rn$,
and $F$ a measurable function on $\rn\times E$.
Then
$$
\lf\|\int_{E}|F(\cdot,y)|\,dy\r\|_{X}
\le|E| \sup_{y\in E}\lf\|F(\cdot,y)\r\|_{X}.
$$
\end{lemma}

\begin{proof}
By Lemma \ref{Lesdual} and the fact that $X$ is a ball Banach function space, we have
\begin{align}\label{min1}
\lf\|\int_{E}|F(\cdot,y)|\,dy\r\|_{X}&=\lf\|\int_{E}|F(\cdot,y)|\,dy\r\|_{X''}\\
&=
\sup\lf\{\lf|\int_\rn\int_{E}|F(x,y)|\,dyg(x)\,dx\r|
:\  g\in X'\text{ such that }\|g\|_{X'}=1\r\}.\noz
\end{align}
From the Tonelli theorem and Lemma \ref{LeHolder}, it follows that,
for any $g\in X'$ such that $\|g\|_{X'}=1$,
\begin{align*}
\lf|\int_\rn\int_{E}F(x,y)\,dy g(x)\,dx\r|
&\le\int_\rn\int_{E}|F(x,y)||g(x)|\,dy\,dx\\
&=\int_{E}\int_\rn|F(x,y)||g(x)|\,dx\,dy\\
&\le\int_{E} \sup_{y\in E} \int_\rn|F(x,y)||g(x)|\,dx\,dy\\
&\ls\int_{E} \sup_{y\in E} \lf\|F(\cdot,y)\r\|_{X}\lf\|g\r\|_{X'}\,dy\\
&\sim |E| \sup_{y\in E}\lf\|F(\cdot,y)\r\|_{X},
\end{align*}
which, together with \eqref{min1}, then
implies the desired inequality.
This finishes the proof of Lemma \ref{LeMI}.
\end{proof}

\begin{definition}\label{def-eNet}
Let $\ez\in(0,\fz)$,
$\cf$ be a subset of the ball Banach function space $X$,
and $\cg\subset\cf$. Then $\cg$ is called an \emph{$\ez$-net} of $\cf$ if, for any $f\in\cf$,
there exists a $g\in\cg$ such that $\|f-g\|_X<\ez$.
Moreover, if $\cg$ is an $\ez$-net of $\cf$ and the cardinality of $\cg$ is finite,
then $\cg$ is called a \emph{finite $\ez$-net} of $\cf$.
Furthermore, $\cf$ is said to be \emph{totally bounded} if, for any $\ez\in(0,\fz)$,
there exists a finite $\ez$-net.
In addition, $\cf$ is said to be \emph{relatively compact} if the closure in $X$ of
$\cf$ is compact.
\end{definition}

From the Hausdorff theorem (see, for instance, \cite[p.\,13, Theorem]{Yosida95}),
it follows that a subset $\cf$ of a ball Banach function space $X$ is relatively compact
if and only if $\cf$ is totally bounded due to the completeness of $X$.

Next, we give a sufficient condition for subsets of ball Banach function spaces
to be totally bounded, which is a generalization in $X$ of
the well-known Fr\'{e}chet--Kolmogorov theorem in $L^p(\rn)$ with $p\in[1,\fz)$.

\begin{theorem}\label{l-fre}
Let $X$ be a ball Banach function space
satisfying the triangle inequality as in \eqref{eq22x}.
Then a subset $\cf$ of
$X$ is totally bounded if
the set $\cf$ satisfies the following three conditions:
\begin{itemize}
\item[{\rm(i)}] $\cf$ is bounded, namely, $$\sup_{f\in\cf}\|f\|_{X}<\infty;$$
\item[{\rm(ii)}] $\cf$ uniformly vanishes at infinity, namely,
for any given $\epsilon\in(0,\infty)$,
there exists a positive constant $M$ such that, for any $f\in\cf$,
$$\lf\|f{\mathbf 1}_{\{x\in\rn:\ |x|>M\}}\r\|_{X}<\epsilon;$$

\item[{\rm(iii)}] $\cf$ is uniformly equicontinuous, namely,
for any given $\epsilon\in(0,\infty)$,
there exists a positive constant $\rho$ such that,
for any $f\in\cf$ and $\xi\in \rn$ with $|\xi|\in[0,\rho)$,
$$\|f(\cdot+\xi)-f(\cdot)\|_{X}<\epsilon.$$
\end{itemize}

Conversely, assume that $X$ satisfies the following additional assumptions that $C_{\rm c}(\rn)$ is dense in
$X$ and, for any $f\in X$ and $y\in\rn$,
\begin{align}\label{+y}
\|f\|_X=\|f(\cdot+y)\|_X.
\end{align}
If a subset $\cf$ of $X$ is totally bounded,
then $\cf$ satisfies (i) through (iii).
\end{theorem}

\begin{proof} We first show the first part of this theorem.
To achieve this, let $\cf\subset X$ satisfy (i), (ii), and (iii).
We now prove that $\cf$ is totally bounded. To this end,
by the fact that $X$ is a Banach space,
it suffices to find a finite $\ez$-net of $\cf$
for any given $\ez\in(0,\fz)$.

For any $i\in\zz$, let ${\mathfrak R}_{i}:=[2^{-i},2^i]^n$.
From (ii), we deduce that there exists an $M\in\nn$
such that, for any $f\in\cf$,
$$\lf\|f-f{\mathbf 1}_{\mathfrak{R}_M}\r\|_{X}<\ez/3.$$
Therefore, to find a finite $\ez$-net of $\cf$,
it suffices to find a finite $2\ez/3$-net of $\{f{\mathbf 1}_{\mathfrak{R}_M}\}_{f\in\cf}$.
To achieve this, we use the following finite dimensional method similar to
that used in \cite{CC13,gz20}.

First, by (iii), we conclude that there exists an $i_\ez\in\zz$ such that,
for any $f\in\cf$ and $\xi\in {\mathfrak R}_{i_\ez}$,
\begin{align}\label{f-f}
\|f(\cdot+\xi)-f(\cdot)\|_{X}<2^{-n}\ez/3.
\end{align}
Observe that, for any $x\in {\mathfrak R}_M$, there exists a unique dyadic cube
$Q_x:=\Pi_{j=1}^n[m_j2^{i_\ez},(m_j+1)2^{i_\ez})$
which has side length $2^{i_\ez}$ and contains $x$ for some integers $m_j$.
For any $f\in\cf$ and $x\in\rn$, let
$$\Phi(f{\mathbf 1}_{\mathfrak{R}_M})(x):=
\begin{cases}
{\displaystyle f_{Q_x}=\frac1{|Q_x|}\int_{Q_x}f(y)\,dy} &{\rm if\ }x\in \mathfrak{R}_M,\\
0 &{\rm if\ }x\in \rn\setminus \mathfrak{R}_M.
\end{cases}$$
We claim that $\Phi$ is well defined, namely,
for any given $f\in\cf$ and any $x\in {\mathfrak R}_M$,
$\Phi(f{\mathbf 1}_{{\mathfrak R}_M})(x)<\fz$.
Indeed, from Lemma \ref{LeMI}, $f\in\cf$, and \eqref{f-f}, it follows that
\begin{align*}
\lf\|\int_{{\mathfrak R}_{i_\ez}}|f(\cdot+\xi)|\,d\xi\r\|_{X}
&\le\lf\|\int_{{\mathfrak R}_{i_\ez}}|f(\cdot+\xi)-f(\cdot)|\,d\xi\r\|_{X}
+|{\mathfrak R}_{i_\ez}|\lf\|f\r\|_X\\
&\le|{\mathfrak R}_{i_\ez}| \sup_{\xi\in {\mathfrak R}_{i_\ez}}\lf\|f(\cdot+\xi)-f(\cdot)\r\|_{X}
+|{\mathfrak R}_{i_\ez}|\lf\|f\r\|_X<\fz.
\end{align*}
By this, we know that, for any given cube $Q_x\subset {\mathfrak R}_M$,
there exists a point $x_0\in Q_x$ such that
$\int_{{\mathfrak R}_{i_\ez}}|f(x_0+\xi)|\,d\xi<\fz$ which,
combined with the observation $Q_x\subset x_0+{\mathfrak R}_{i_\ez}$,
further implies that
$$\int_{Q_x}|f(\xi)|\,d\xi
\le\int_{x_0+{\mathfrak R}_{i_\ez}}|f(\xi)|\,d\xi
=\int_{{\mathfrak R}_{i_\ez}}|f(x_0+\xi)|\,d\xi<\fz.$$
This shows that the above claim holds true.

Now, we estimate $\|f{\mathbf 1}_{{\mathfrak R}_M}-\Phi(f{\mathbf 1}_{{\mathfrak R}_M})\|_X$
for any given $f\in \cf$.
To this end, notice that, for any $x\in {\mathfrak R}_M$,
\begin{align*}
\lf|[f(x)-f_{Q_x}]{\mathbf 1}_{Q_x}(x) \r|
&= \lf|\frac1{|Q_x|}\int_{Q_x}[f(x)-f(y)]\,dy {\mathbf 1}_{Q_x}(x)  \r| \\
&\le \frac1{|Q_x|}\int_{Q_x}\lf|f(x)-f(y) \r|\,dy {\mathbf 1}_{Q_x}(x) \\
&\le \frac1{|Q_x|}\int_{{\mathfrak R}_{i_\ez}}\lf|f(x)-f(x+\xi) \r|\,d\xi {\mathbf 1}_{Q_x}(x).
\end{align*}
From this, we deduce that, for any $x\in {\mathfrak R}_M$,
\begin{align*}
\lf|\sum_{Q_x\subset {\mathfrak R}_M}[f(x)-f_{Q_x}]{\mathbf 1}_{Q_x}(x) \r|
&\le 2^{-n i_\ez}\int_{{\mathfrak R}_{i_\ez}}\lf|f(x)-f(x+\xi) \r|\,d\xi {\mathbf 1}_{{\mathfrak R}_M}(x)
\end{align*}
and hence, by Lemma \ref{LeMI} and \eqref{f-f}, we have
\begin{align}\label{ez/3}
\lf\|f{\mathbf 1}_{{\mathfrak R}_M}-\Phi(f{\mathbf 1}_{{\mathfrak R}_M}) \r\|_X
&=\lf\|\sum_{Q_x\subset {\mathfrak R}_M}[f(\cdot)-f_{Q_x}]{\mathbf 1}_{Q_x}(\cdot) \r\|_X \\
&\le 2^{-n i_\ez} \lf\|\int_{{\mathfrak R}_{i_\ez}}\lf|f(\cdot)-f(\cdot+\xi) \r|\,d\xi
 {\mathbf 1}_{{\mathfrak R}_M}(\cdot) \r\|_X\noz\\
&\le 2^{-n i_\ez} |{\mathfrak R}_{i_\ez}|\sup_{\xi\in {\mathfrak R}_{i_\ez}}\lf\|f(\cdot)-f(\cdot+\xi) \r\|_X \noz\\
&=2^n\sup_{\xi\in {\mathfrak R}_{i_\ez}}\lf\|f(\cdot)-f(\cdot+\xi) \r\|_X
<\ez/3.\noz
\end{align}
This is the desired estimate.

In addition, it is easy to see that $\{\Phi(f{\mathbf 1}_{{\mathfrak R}_M})\}_{f\in\cf}$
is a bounded subset of a finite dimensional Banach space,
which implies that $\{\Phi(f{\mathbf 1}_{{\mathfrak R}_M})\}_{f\in\cf}$ has a finite $\ez/3$-net.
This, together with \eqref{ez/3},
shows that there exists a finite $2\ez/3$-net of $\{f{\mathbf 1}_{{\mathfrak R}_M}\}_{f\in\cf}$,
which completes the proof of the first part of this theorem.

Now, we show the second part of this theorem.
Let $\cf$ be a totally bounded set of $X$.
Then, by the definition of totally bounded sets, we easily know
that (i) holds true.

Next, for any given $\ez\in(0,\fz)$,
let $\{U_1,\dots,U_m\}$ be a finite $\ez/3$-net of $\cf$,
and choose $g_j\in U_j$ for any $j\in\{1,\dots,m\}$.
By the additional assumption that $C_{\rm c}(\rn)$ is dense in $X$,
we may assume that $g_j$ also belongs to $C_{\rm c}(\rn)$
for any $j\in\{1,\dots,m\}$,
which implies that there exists a positive constant $M$ such that,
for any $j\in\{1,\dots,m\}$,
$$\lf\|g_j\mathbf1_{\{x\in\rn:\ |x|>M\}}\r\|_X=0.$$
Thus, for any given $j\in\{1,\dots,m\}$, if $f\in U_j$,
then $\|f-g_j\|_X<\ez/3$ and hence
\begin{align*}
\lf\|f{\mathbf 1}_{\{x\in\rn:\ |x|>M\}}\r\|_{X}
&\le\lf\|(f-g_j){\mathbf 1}_{\{x\in\rn:\ |x|>M\}}\r\|_{X}
+\lf\|g_j{\mathbf 1}_{\{x\in\rn:\ |x|>M\}}\r\|_{X}\\
&\le\lf\|f-g_j\r\|_{X}
+\lf\|g_j{\mathbf 1}_{\{x\in\rn:\ |x|>M\}}\r\|_{X}
<\ez/3<\ez.
\end{align*}
This shows that (ii) holds true.

Finally, for any given $j\in\{1,\dots,m\}$, by $g_j\in C_{\rm c}(\rn)$,
we conclude that there exists a positive constant $\rho$ such that,
for any $\xi\in\rn$ with $|\xi|<\rho$,
\begin{align}\label{g-g}
\lf\|g_j(\cdot+\xi)-g_j(\cdot)\r\|_{X}<\ez/3.
\end{align}
Moreover, for any given $f\in\cf$,
there exists a $g_j\in U_j\cap C_{\rm c}(\rn)$ with some $j\in\{1,\dots,m\}$
such that $\|f-g_j\|_X<\ez$,
which, combined with \eqref{+y} and \eqref{g-g}, further implies that,
for any $\xi\in\rn$ with $|\xi|<\rho$,
\begin{align*}
\lf\|f(\cdot+\xi)-f(\cdot)\r\|_{X}
\le\lf\|f(\cdot+\xi)-g_j(\cdot+\xi)\r\|_{X}
+\lf\|g_j(\cdot+\xi)-g_j(\cdot)\r\|_{X}
+\lf\|g_j-f\r\|_{X}
<\ez.
\end{align*}
This shows that (iii) holds true, which completes the proof of
the second part of this theorem and hence of Theorem \ref{l-fre}.
\end{proof}

\begin{remark}\label{rem-+y}
\begin{itemize}
\item [{\rm(i)}] In the second part of Theorem \ref{l-fre}, the
additional assumption \eqref{+y} is reasonable because,
even when $X$ is the weighted Lebesgue space, if $f\in X$,
then $f(\cdot+\xi)$ may not be in $X$ even when $|\xi|$ is small.

\item [{\rm(ii)}]
If $X$ has an absolutely continuous quasi-norm,
then $C_{\rm c}(\rn)$ is dense in $X$ [see Proposition \ref{2fre} below].
Recall that a ball quasi-Banach function space $X$ is said to have an
\emph{absolutely continuous quasi-norm} if, for any $f\in X$
and any sequence of measurable sets, $\{E_j\}_{j\in\nn}\subset \rn$,
satisfying that $\mathbf{1}_{E_j}\to 0$ almost everywhere as $j\to\fz$,
$\|f\mathbf{1}_{E_j}\|_X\to 0$ as $j\to\fz$.
\end{itemize}
\end{remark}

\begin{proposition}\label{2fre}
Let $X$ be a ball quasi-Banach function space
having an absolutely continuous quasi-norm. Then
$C_{\mathrm c}(\rn)$ is dense in $X$.
\end{proposition}

\begin{proof}
Without loss of generality, we may
let $f\in X$ be a non-negative measurable function on $\rn$.
Then there exists an increasing sequence of non-negative simple functions
$\{f_j\}_{j\in\nn}$
that converges pointwise to $f$.
From this and \cite[p.\,16, Proposition 3.6]{BS}, it follows that,
for any given $\varepsilon\in(0,\infty)$,
there exists a simple function $g:=\sum_{k=1}^{N}\lambda_k\mathbf{1}_{E_k}\in \{f_j\}_{j\in\nn}$
such that $\|f-g\|_{X}<\varepsilon$,
where, for any $k\in\{1,\ldots,N\}$,
$E_k$ is a measurable set and $\lambda_k$ is a positive
constant.
Further, by
the inner regularity of the Lebesgue measure and \cite[p.\,16, Proposition 3.6]{BS},
we know that there exists a simple function $h:=\sum_{k=1}^{N}\lambda_k\mathbf{1}_{F_k}$
such that $\|g-h\|_{X}<\varepsilon$,
where, for any $k\in\{1,\ldots,N\}$,
$F_k\subset E_k$ is a compact set, which,
together with the outer regularity of the Lebesgue measure
and \cite[p.\,16, Proposition 3.6]{BS}, further implies that
there exists a simple function $u:=\sum_{k=1}^{N}\lambda_k\mathbf{1}_{U_k}$
such that $\|u-h\|_{X}<\varepsilon$,
where, for any $k\in\{1,\ldots,N\}$,
$F_k\subset U_k$ is a bounded open set.
Then, using the Urysohn lemma, we obtain $f_0\in C_{\mathrm c}(\rn)$
satisfying that $0\le f_0-h\le u-h$ and hence $\|f_0-h\|_{X}<\|u-h\|_{X}<\varepsilon$.
By the above estimates, we conclude that
$\|f-f_0\|_{X}<\varepsilon$. This finishes the proof of Proposition \ref{2fre}.
\end{proof}

Let $\Omega\in L^\infty(\mathbb{S}^{n-1})$ satisfy \eqref{deg0},
\eqref{mean0}, and $L^\infty$-Dini condition,
and $T_\Omega$ be
a singular integral operator with homogeneous kernel $\Omega$.
To show Theorem \ref{thm-cpt},
we first establish the boundedness of the maximal operator $T_\Omega^\ast$
of a family of truncated transforms
$\{T_{\Omega,\eta}\}_{\eta\in(0,\infty)}$ defined as follows.
For any given $\eta\in(0,\infty)$ and for any $f\in X$
and $x\in\rn$, let
$$T_{\Omega,\eta} f(x):=
\int_{\{y\in\rn:\ |x-y|>\eta\}}
\frac{\Omega(x-y)}{|x-y|^n}
f(y)\,dy.$$
The\emph{ maximal operator $T_\Omega^\ast$} is
defined by setting, for any $f\in X$ and $x\in\rn$,
\begin{equation}\label{ssd}
T_\Omega^\ast f(x):=\sup_{\eta\in(0,\infty)}\lf|T_{\Omega,\eta} f(x)\r|
=\sup_{\eta\in(0,\infty)}
\lf|\int_{\{y\in\rn:\ |x-y|>\eta\}}\frac{\Omega(x-y)}{|x-y|^n}
f(y)\,dy\r|.
\end{equation}
We point out that Proposition \ref{pp6} below
ensures that, for any $f\in X$, $T_\Omega^\ast f$ in \eqref{ssd}
is well defined.

Recall that the following weighted
$L^p(\rn)$ boundedness of the maximal operator $T_\Omega^\ast$ is a part of
\cite[Theorem 2.1.8]{ldy}.

\begin{lemma}\label{coa5}
Let $\omega\in A_1(\rn)$, $p\in(1,\infty)$, and
$\Omega\in L^\infty(\mathbb{S}^{n-1})$ satisfy \eqref{deg0}, \eqref{mean0},
and the $L^\infty$-Dini condition. Assume that $T_\Omega^\ast$ is
a singular integral operator with homogeneous kernel $\Omega$.
Then there exists
a positive constant $C_{(p,\Omega,[\oz]_{A_1(\rn)})}$ such that,
for any $f\in L^p_\oz(\rn)$,
$$
\int_\rn [T_\Omega^\ast(f)(x)]^p\omega(x)\,dx\leq
C_{(p,\Omega,[\oz]_{A_1(\rn)})}\int_\rn\lf|f(x)\r|^p\omega(x)\,dx.
$$
\end{lemma}

As an immediate consequence of Lemma \ref{coa5}, we have the following conclusion.

\begin{proposition}\label{pp6}
Let $X$ be a ball quasi-Banach function space satisfying Assumption \ref{assum}(ii),
$\Omega\in L^\infty(\mathbb{S}^{n-1})$ satisfy \eqref{deg0}, \eqref{mean0}, and \eqref{Dini},
and $T_\Omega^\ast$ be the maximal operator as in \eqref{ssd}.
Then there exists a positive constant $C$ such that,
for any $f\in X$,
\begin{equation}\label{Eqcc2}
\lf\|T_\Omega^\ast(f)\r\|_X\le C\lf\|f\r\|_X,
\end{equation}
and, for any $f\in X$ and almost every $x\in\rn$,
\begin{equation}\label{Eqcc3}
T_\Omega(f)(x)=\lim_{\varepsilon\to0^+}\int_{\varepsilon<|x-y|<1/\varepsilon}\frac{\Omega(x-y)}{|x-y|^n}f(y)\,dy.
\end{equation}
\end{proposition}

\begin{proof}
Using Lemma \ref{coa5} and Proposition \ref{thm-bdd-0},
we immediately obtain \eqref{Eqcc2}.
Moreover, from Lemma \ref{embed1}, we deduce that
$X\subset L_\omega^s(\rn)$.
By
Lemma \ref{coa5} and \cite[Theorem 2.2]{d}, we know that,
for any $f\in L_\omega^s(\rn)$ and almost every $x\in\rn$,
\begin{equation*}
T_\Omega(f)(x)=\lim_{\varepsilon\to0^+}
\int_{\varepsilon<|x-y|<1/\varepsilon}\frac{\Omega(x-y)}{|x-y|^n}f(y)\,dy
\end{equation*}
and hence \eqref{Eqcc3} holds true.
This finishes the proof of Proposition \ref{pp6}.
\end{proof}

Next, we recall the following smooth truncated technique same
as in \cite{CC13} (see also \cite{KL01}).
Let $\varphi\in C^\fz([0,\fz))$ satisfy
$$0\le\varphi\le1{\quad \rm and\quad }
\varphi(x)=
\begin{cases}
1, &x\in[0,1/2],\\
0, &x\in[1,\fz].
\end{cases}$$
Let
$\Omega\in L^\infty(\mathbb{S}^{n-1})$ satisfy \eqref{deg0}, \eqref{mean0}, and
the $L^\infty$-Dini condition.
For any $\kappa\in(0,\fz)$ and any $x,\ y\in\rn$, define
$K_\kappa(x,y):=\frac{\Omega(x-y)}{|x-y|^n}[1-\varphi(\frac{|x-y|}{\kappa})]$.
Let $X$ be a ball quasi-Banach function space.
Assume that there exists an $s\in(0,\infty)$ such that $X^{1/s}$ is a ball Banach function space
and $\cm$ bounded on $(X^{1/s})'$.
By Lemma \ref{embed1} and \cite[Lemma 7.4.5]{G1}, we know that,
for any $f\in X$, $\kappa\in(0,\fz)$, and $x\in\rn$,
\begin{align*}
T_\Omega^{(\kappa)}f(x):=\int_{\rn}K_\kappa(x,y)f(y)\,dy<\infty.
\end{align*}

\begin{remark}\label{Keta-standard}
Let $\Omega\in L^\infty(\mathbb{S}^{n-1})$ satisfy \eqref{deg0},
\eqref{mean0}, and \eqref{Dini}.
Then, for any given $\kappa\in(0,\fz)$,
$K_\kappa$ satisfies the following smoothness condition:
there exists a positive constant $C$ such that, for any
$x,\ y,\ \xi\in\rn$ with $|\xi|\le|x-y|/2$,
\begin{align}\label{Kk}
\lf|K_\kappa(x,y)-K_\kappa(x+\xi,y) \r|
\le C\lf[\frac{1}{|x-y|^{n}}\omega_\infty\lf(\frac{4|\xi|}{|x-y|}\r)+
\frac{|\xi|}{|x-y|^{n+1}}\r].
\end{align}
Indeed, by \eqref{deg0}, we conclude that,
for any $x,\ y,\ \xi\in\rn$ with $|\xi|\le|x-y|/2$,
\begin{align*}
\lf|\Omega(x-y)
-\Omega(x+\xi-y)\r|=\lf|\Omega\lf(\frac{x-y}{|x-y|}\r)
-\Omega\lf(\frac{x+\xi-y}{|x+\xi-y|}\r)\r|
\le\omega_\infty\lf(\frac{4|\xi|}{|x-y|}\r)
\end{align*}
From this, $\Omega\in L^\fz(\rn)$, the mean value theorem, and
the definition of $\varphi$, it follows that,
for any $x,\ y,\ \xi\in\rn$ with $|\xi|\le|x-y|/2$,
\begin{align*}
&\lf|K_\kappa(x,y)-K_\kappa(x+\xi,y) \r|\\
&\quad\le\lf|\frac{\Omega(x-y)}{|x-y|^n}
-\frac{\Omega(x+\xi-y)}{|x+\xi-y|^n}\r|\lf|1-\varphi\lf(\frac{|x-y|}{\kappa}\r) \r|
+\lf|\frac{\Omega(x+\xi-y)}{|x+\xi-y|^n}\r|\\
&\quad\quad\times
\lf|\varphi\lf(\frac{|x+\xi-y|}{\kappa}\r)-\varphi\lf(\frac{|x-y|}{\kappa}\r) \r|\\
&\quad\ls\lf|\frac{\Omega(x-y)
-\Omega(x+\xi-y)}{|x-y|^n}\r|
+\lf|\frac{1}{|x+\xi-y|^n}-\frac{1}{|x-y|^n}\r|\\
&\quad\quad+\frac{\|\varphi'\|_{L^\fz(\rr_+)}}{|x+\xi-y|^n}
\lf|\frac{|x+\xi-y|}{\kappa}-\frac{|x-y|}{\kappa} \r|
\mathbf1_{\{(x,y)\in\rn\times\rn:\ \frac13\kappa\le|x-y|\le2\kappa\}}(x,y)\\
&\quad\ls\frac{1}{|x-y|^{n}}\omega_\infty\lf(\frac{4|\xi|}{|x-y|}\r)
+\frac{|\xi|}{|x-y|^{n+1}}+\frac1\kappa\frac{|\xi|}{|x-y|^{n}}
\mathbf1_{\{(x,y)\in\rn\times\rn:\ \frac13\kappa\le|x-y|\le2\kappa\}}(x,y)\\
&\quad\ls\frac{1}{|x-y|^{n}}\omega_\infty\lf(\frac{4|\xi|}{|x-y|}\r)+
\frac{|\xi|}{|x-y|^{n+1}},
\end{align*}
where the implicit positive constants are independent of $\kappa$, $x$, $y$, and $\xi$.
\end{remark}

\begin{lemma}\label{Teta}
Let $b\in C_{\rm c}^\fz(\rn)$,
$X$ be a ball quasi-Banach function space satisfying Assumption \ref{assum}(ii),
and $\Omega\in L^\infty(\mathbb{S}^{n-1})$ satisfy \eqref{deg0}, \eqref{mean0}, and
\eqref{Dini}.
Then there exists a positive constant $C$ such that,
for any $\kappa\in(0,\fz)$, $f\in X$, and $x\in X$,
$$\lf|[b,T_\Omega](f)(x)- [b,T_\Omega^{(\kappa)}](f)(x)\r|
\le C \kappa \|\nabla b\|_{L^\fz(\rn)}\cm f(x).$$
Moreover, if $\cm$ is bounded on $X$, then
$$\lim_{\kappa\to0^+}\lf\|[b,T_\Omega]- [b,T_\Omega^{(\kappa)}]\r\|_{X\to X}=0.$$
\end{lemma}

\begin{proof}
Let $f\in X$. For any $x\in\rn$, by the mean value theorem and \eqref{Eqcc3}, we have
\begin{align*}
&\lf|[b,T_\Omega](f)(x)- [b,T_\Omega^{(\kappa)}](f)(x)\r|\\
&\quad=\lf|\lim_{\varepsilon\to0^+}\int_{\varepsilon<|x-y|<1/\varepsilon}
[b(x)-b(y)]\frac{\Omega(x-y)}{|x-y|^n}
\varphi\lf(\frac{|x-y|}{\kappa}\r)f(y)\,dy \r|\\
&\quad\le\int_{\{y\in\rn:\ |x-y|\le\kappa\}}|b(x)-b(y)|\frac{|\Omega(x-y)|}{|x-y|^n}|f(y)|\,dy\\
&\quad\le\|\Omega\|_{L^\fz(\rn)}\|\nabla b\|_{L^\fz(\rn)}\sum_{j=0}^\fz
 \int_{\{y\in\rn:\ \frac{\kappa}{2^{j+1}}<|x-y|\le\frac{\kappa}{2^{j}}\}}
 |x-y|\frac{|f(y)|}{|x-y|^{n}}\,dy\\
&\quad\le\|\Omega\|_{L^\fz(\rn)}\|\nabla b\|_{L^\fz(\rn)}\sum_{j=0}^\fz
 \frac{\kappa}{2^{j}}\lf(\frac{\kappa}{2^{j+1}}\r)^{-n}
 \int_{B(x,\frac{\kappa}{2^j})}|f(y)|\,dy\\
&\quad\le\kappa2^n \|\Omega\|_{L^\fz(\rn)}\|\nabla b\|_{L^\fz(\rn)}
 \sum_{j=0}^\fz2^{-j} \cm f(x)\lesssim\kappa\|\nabla b\|_{L^\fz(\rn)}\cm f(x).
\end{align*}
Moreover, if $\cm$ is bounded on $X$, then
$$\lf\|[b,T_\Omega](f)- [b,T_\Omega^{(\kappa)}](f) \r\|_X
\lesssim \kappa\|\cm f\|_X
\lesssim \kappa\|f\|_X,$$
which implies that
$\lim_{\kappa\to0^+}\|[b,T_\Omega]- [b,T_\Omega^{(\kappa)}]\|_{X\to X}=0$
and hence completes the proof of Lemma \ref{Teta}.
\end{proof}

\begin{proof}[Proof of Theorem \ref{thm-cpt}]
Let $b\in {\rm CMO}(\rn)$.
By the definition of ${\rm CMO}(\rn)$, we know that,
for any given $\varepsilon\in(0,\infty)$,
there exists a $b^{(\varepsilon)}\in C_{\rm c}^\fz(\rn)$ such that
$\|b-b^{(\varepsilon)}\|_{{\rm BMO}(\rn)}<\varepsilon.$
Then, by the boundedness of $[b-b^{(\varepsilon)}, T_\Omega]$ on $X$
(see Theorem \ref{thm-bdd-1}), we obtain,
for any given $\varepsilon\in(0,\infty)$ and
for any $f\in X$,
\begin{align*}
&\lf\|[b,T_\Omega](f)-[b^{(\varepsilon)},T_\Omega](f)\r\|_{X}\\
&\quad=\lf\|[b-b^{(\varepsilon)},T_\Omega](f)\r\|_{X}
\lesssim\lf\|b-b^{(\varepsilon)}\r\|_{{\rm BMO}(\rn)}\|f\|_{X}
\ls\varepsilon\|f\|_{X}.
\end{align*}
From this, Lemma \ref{Teta}, and the fact that the limit of compact operators
is also a compact operator, it follows that, to prove Theorem \ref{thm-cpt},
it suffices to show that, for any $b\in C_{\rm c}^\fz(\rn)$
and any $\kappa\in(0,\fz)$ small enough,
$[b,\,T_\Omega^{(\kappa)}]$ is a compact operator on $X$.
To this end, by the definition of compact operators,
it suffices to prove that, for any bounded subset
$\cf\subset X$,
$$[b,\,T_\Omega^{(\kappa)}]\cf:=\lf\{[b,\,T_\Omega^{(\kappa)}](f):\ f\in\cf\r\}$$
is relatively compact.
To achieve this, from Theorem \ref{l-fre}, we deduce that it suffices to
show that $[b,\,T_\Omega^{(\kappa)}]\cf$ satisfies
the conditions (i) through (iii) of Theorem \ref{l-fre}
for any given $b\in C_{\rm c}^\fz(\rn)$ and $\kappa\in(0,\fz)$ small enough.

By Theorem \ref{thm-bdd-1} and Lemma \ref{Teta},
we conclude that $[b,\,T_\Omega^{(\kappa)}]$ is bounded on $X$
for any given $\kappa\in(0,\fz)$, which implies that
$[b,\,T_\Omega^{(\kappa)}]\cf$ satisfies the condition (i) of Theorem \ref{l-fre}.

Next, since $b\in C_{\rm c}^\fz(\rn)$, it follows that
there exists a positive constant $R_0$ such that $\supp (b)\subset B(\vec{0}_n,R_0)$.
Let $M\in(2R_0,\infty)$. Then,
for any $y\in B(\vec{0}_n,R_0)$ and $x\in\rn$ with $|x|\in(M,\infty)$,
we have $|x-y|\sim |x|$.
Moreover, by this, $\Omega\in L^\infty(\mathbb{S}^{n-1})$,
and Lemma \ref{LeHolder},
we conclude that, for any $f\in\cf$ and
$x\in\rn$ with $|x|\in(M,\infty)$,
\begin{align*}
\lf|[b,\,T_\Omega^{(\kappa)}](f)(x)\r|&\le\int_\rn|b(x)-b(y)|
\frac{|\Omega(x-y)|}{|x-y|^n}
|f(y)|\,dy
\ls\int_{B(\vec{0}_n,R_0)}
\frac{|f(y)|}{|x|^n}\,dy\\
&\ls\frac{1}{|x|^n}\|f\|_{X}\lf\|{\mathbf 1}_{B(\vec{0}_n,R_0)}\r\|_{X'}
\ls\frac{1}{|x|^n}.
\end{align*}
From this and Lemma \ref{as1} with $\tz$ replaced by $\eta\in(1,\fz)$ in Remark \ref{eta},
we deduce that
\begin{align*}
&\lf\|[b,\,T_\Omega^{(\kappa)}](f){\mathbf 1}_{\{x\in\rn:\ |x|>M\}}\r\|_{X}\noz\\
&\quad\ls
\sum_{j=0}^\infty
\lf\|\frac1{|\cdot|^n} {\mathbf 1}_{\{x\in\rn:\ 2^jM<|x|\leq 2^{j+1}M\}}\r\|_{X}\ls
\sum_{j=0}^\infty
\frac{\|{\mathbf 1}_{\{x\in\rn:\ |x|\leq 2^{j+1}M\}}\|_{X}}{(2^jM)^{n}}\noz\\
&\quad\ls
\sum_{j=0}^\infty
\frac{(2^{j+1}M)^{n/\eta}}{(2^jM)^{n}}
\ls
\sum_{j=0}^\infty
\frac{1}{(2^jM)^{n(1-1/\eta)}}\lesssim
\frac{1}{M^{n(1-1/\eta)}}.
\end{align*}
Therefore, the condition (ii) of Theorem \ref{l-fre} holds true for
$[b,T_\Omega^{(\kappa)}]\mathcal{F}$.

It remains to prove that $[b,T_\Omega^{(\kappa)}]\mathcal{F}$
also satisfies the condition (iii) of Theorem \ref{l-fre}.
For any $f\in\cf$,
$\xi\in\rn\setminus\{\vec{0}_n\}$,
and $x\in\rn$, we have
\begin{align}\label{L1+L2}
&[b, T_\Omega^{(\kappa)}](f)(x)-[b, T_\Omega^{(\kappa)}](f)(x+\xi)\\
&\quad=\int_{\rn}[b(x)-b(y)] K_\kappa(x,y) f(y)\,dy
 -\int_{\rn}[b(x+\xi)-b(y)] K_\kappa(x+\xi,y) f(y)\,dy\notag\\
&\quad=[b(x)-b(x+\xi)]\int_{\rn} K_\kappa(x,y) f(y)\,dy\notag\\
&\qquad+\int_{\rn}[b(x+\xi)-b(y)] \lf[K_\kappa(x,y)-K_\kappa(x+\xi,y)\r] f(y)\,dy\notag\\
&\quad=:{\rm L}_1(x)+{\rm L}_2(x)\noz.
\end{align}
We first estimate L$_1(x)$. Observe that, by the mean value theorem and the definition
of $K_\kappa$,
\begin{align*}
\lf|{\rm L}_1(x)\r|&\le \xi\|\nabla b\|_{L^\fz(\rn)}
\lf|\int_{\{y\in\rn:\ |x-y|\ge\frac\kappa2\}}\lf[K_\kappa(x,y)
 -\frac{\Omega(x-y)}{|x-y|^n}\r]f(y)\,dy\r. \\
&\qquad\qquad\qquad\qquad
 \lf.+\int_{\{y\in\rn:\|x-y|\ge\frac\kappa2\}}\frac{\Omega(x-y)}{|x-y|^n}f(y)\,dy\r|\\
&\lesssim\xi
 \lf[\int_{\{y\in\rn:\ \kappa\ge|x-y|\ge\frac\kappa2\}}
 \lf|\frac{\Omega(x-y)}{|x-y|^n} \r||f(y)|\,dy
 +T_\Omega^\ast f(x) \r]\\
&\ls\xi\lf[\cm f(x)+T_\Omega^\ast f(x) \r],
\end{align*}
where the implicit constant is independent of $f$, $\xi$, and $x$.
From this, the boundedness of $\cm$ on $X$, and Proposition \ref{pp6},
we deduce that
\begin{align}\label{L1}
\|{\rm L}_1\|_X\ls |\xi|  \|f\|_X.
\end{align}
Now, we estimate L$_2(x)$. Observe that, for any $x,\ y,\ \xi\in\rn$ with
$|x-y|<\kappa/4$ and $|\xi|<\kappa/8$, we have
$|x-y|/\kappa<1/2$ and $|x+\xi-y|/\kappa<1/2$, which implies that
$\varphi(|x-y|/\kappa)=0=\varphi(|x+\xi-y|/\kappa)$
and hence
\begin{align}\label{K=0=K}
K_\kappa(x,y)=0=K_\kappa(x+\xi,y).
\end{align}
Besides,  for any $x,\ y,\ \xi\in\rn$ with $|x-y|\ge\kappa/4$ and $|\xi|<\kappa/8$,
we have $|\xi|\le|x-y|/2$.
From this, \eqref{Kk}, and \eqref{K=0=K}, we deduce that, for any given $\xi\in\rn$ with $|\xi|<\kappa/8$,
\begin{align*}
|{\rm L}_2(x)|&\ls
|\xi|
\int_{\{y\in\rn:\ |x-y|\ge\frac\kappa4\}}\frac{|f(y)|}{|x-y|^{n+1}}\,dy
+\int_{\{y\in\rn:\ |x-y|\ge\frac\kappa4\}}\frac{|f(y)|}{|x-y|^{n}}
\omega_\infty\lf(\frac{4|\xi|}{|x-y|}\r)\,dy\\
&\ls|\xi|
\sum_{k=0}^\fz\lf(2^k\kappa\r)^{-(n+1)} \int_{\{y\in\rn:\ 2^k\frac\kappa4\le|x-y|<2^{k+1}\frac\kappa4\}}
|f(y)|\,dy\\
&\quad+
\sum_{k=0}^\fz\lf(2^k\kappa\r)^{-n}
\omega_\infty\lf(\frac{|\xi|}{2^{k-4}\kappa}\r)
\int_{\{y\in\rn:\ 2^k\frac\kappa4\le|x-y|<2^{k+1}\frac\kappa4\}}
|f(y)|\,dy\\
&\ls\lf[|\xi|+\sum_{k=0}^\fz
\omega_\infty\lf(\frac{|\xi|}{2^{k-4}\kappa}\r)
\int_{2^{-(k+1)}}^{2^{-k}}\,\frac{d\tau}{\tau}\r]
\cm f(x)\\
&\ls\lf[|\xi| +
\int_{0}^{1}
\omega_\infty\lf(\frac{32|\xi|}{\kappa}\tau\r)
\,\frac{d\tau}{\tau}\r]
\cm f(x)\sim\lf[|\xi| +
\int_{0}^{\frac{32|\xi|}{\kappa}}
\omega_\infty\lf(\tau\r)
\,\frac{d\tau}{\tau}\r]
\cm f(x)
\end{align*}
and hence
\begin{align}\label{L2}
\|{\rm L}_2\|_X\ls \lf[|\xi| +
\int_{0}^{\frac{32|\xi|}{\kappa}}
\omega_\infty\lf(\tau\r)
\,\frac{d\tau}{\tau}\r] \|f\|_X.
\end{align}
Combining \eqref{L1+L2}, \eqref{L1}, \eqref{L2},
and the $L^\infty$-Dini condition, we have
$$\lim_{|\xi|\to0^+}
\lf\|[b, T_\Omega^{(\kappa)}](f)(\cdot+\xi)-[b, T_\Omega^{(\kappa)}](f)(\cdot) \r\|_X=0,$$
which implies the condition (iii) of Theorem \ref{l-fre}.
Thus, $[b,\,T_\Omega^{(\kappa)}]$ is a compact operator
for any given $b\in C_{\rm c}^\fz(\rn)$ and $\kappa\in(0,\fz)$.
This finishes the  proof of Theorem \ref{thm-cpt}.
\end{proof}

\subsection{Proof of Theorem \ref{thm-cpt2}}\label{s3.2}

We begin with recalling the following equivalent characterization of ${\rm CMO}(\rn)$
in terms of the local mean oscillation, which is just \cite[Theorem 3.3]{gwy20}.
\begin{lemma}\label{l-cmo char}
Let $f\in{\rm BMO}(\rn)$ and $\lz\in(0,1/2)$.
Then $f\in{\rm CMO}(\rn)$ if and only if $f$ satisfies the following three conditions:
\begin{itemize}
\item [{\rm(i)}]
$\lim_{a\to0^+}\sup_{|B|=a}\oz_{\lz}(f;B)=0$;

\item [{\rm(ii)}]
$\lim_{a\to\fz}
      \sup_{|B|=a}\oz_{\lz}(f;B)=0$;

\item [{\rm (iii)}]
$\lim_{d\to\fz} \sup_{B\cap B(\vec{0}_n,d)=\emptyset}\oz_{\lz}(f;B)=0$,
\end{itemize}
where the local mean oscillation $\oz_{\lz}(f;B)$ is as in \eqref{z3}.
\end{lemma}

To prove Theorem \ref{thm-cpt2},
we establish the lower and the upper estimates of commutators on $X$, respectively,
in Theorems \ref{thmm2} and \ref{thmm3} below.
\begin{proposition}\label{thmm2}
Let $b\in L_{\loc}^1(\rn)$, $\lambda\in(0,1)$,
and $X$ be a ball Banach function space.
Assume that $\cm$ is bounded on $X$ and
$\Omega\in L^\infty(\mathbb{S}^{n-1})$ satisfies that there exists an open set $\Lambda\subset \mathbb{S}^{n-1}$ such that
$\Omega$ does not change sign on $\Lambda$.
Let $B:=B(x_0,r_0)$, $k_0$, $\varepsilon_0$, $E$, and $F$ be as in Lemma \ref{thmb3}.
Then there exists a positive constant $C_{(\lz,k_0,\varepsilon_0,n)}$,
depending only on $\lz$, $k_0$, $\varepsilon_0$, and $n$,
such that, for any measurable set $Q\subset \rn$ with
$|Q |\le\frac{\lambda}{8}|B(x_0,r_0)|$,
$$
\oz_{\lz}(b;B)\|\mathbf{1}_{F}\|_{X}
\le C_{(\lz,k_0,\varepsilon_0,n)}
\lf\|[b,T_\Omega](\mathbf{1}_{F})
\mathbf{1}_{E\setminus Q}\r\|_{X}.
$$
\end{proposition}

\begin{proof}
Let
$b\in L_{\loc}^1(\rn)$, $\lambda\in(0,1)$, and
$B:=B(x_0,r_0)$ with $x_0\in\rn$ and $r_0\in(0,\infty)$;
let
$\varepsilon_0$, $k_0$, $G$, $E$, and $F$
be as in Lemma \ref{thmb3};
and let
$Q$ be a measurable set in $\rn$ with
$|Q |\le\frac{\lambda}{8}|B(x_0,r_0)|$.
Then, by (i) and (iii) of Lemma \ref{thmb3},
we conclude that
$$
\omega_{\lambda}(b;B(x_0,r_0))
\lf|[(E\setminus Q)\times F]\cap G\r|\le
\frac{1}{\varepsilon_0}
\int_{[(E\setminus Q)\times F]\cap G}|b(x)-b(y)|
\lf|\Omega\lf(\frac{x-y}{|x-y|}\r)\r|\,dx\,dy.
$$
From this, the fact that $|x-y|\le2(k_0+1)r_0$
for any $(x,y)\in G$, Lemma \ref{thmb3}(ii), the observations
\begin{align*}|[(E\setminus Q)\times F]\cap G|
&\geq|G|-|Q||F|\\
&\ge\frac{\lambda}{8}|B(x_0,r_0)|^2-\frac{\lambda}{8}|B(x_0,r_0)|\frac{|B(x_0,r_0)|}{2}
=\frac{\lambda}{16}|B(x_0,r_0)|^2
\end{align*}
as well as $\overline{E}\cap \overline{F}=\emptyset$,
and the definition of $[b,T_\Omega]$,
we deduce that
\begin{align*}
\omega_{\lambda}(b;B(x_0,r_0))&\le
\frac{[2(k_0+1)r_0]^n}{\varepsilon_0
|[(E\setminus Q)\times F]\cap G|}\int_{[(E\setminus Q)\times F]\cap G}
\frac{|b(x)-b(y)|}{|x-y|^n}
\lf|\Omega\lf(\frac{x-y}{|x-y|}\r)\r|\,dx\,dy\\
&\le
\frac{16[2(k_0+1)r_0]^n}{\lambda\varepsilon_0|B(x_0,r_0)|^2}
\int_{E\setminus Q}\lf|\int_F
\frac{b(x)-b(y)}{|x-y|^n}
\Omega\lf(\frac{x-y}{|x-y|}\r)\,dy\r|\, dx\\
&\lesssim
\frac{1}{|B(x_0,r_0)|}
\int_{E\setminus Q}\lf|[b,T_\Omega](\mathbf{1}_F)\r|\, dx,
\end{align*}
which, combined with $F\subset 4k_0B(x_0,r_0)$ and Lemmas \ref{LeHolder} and \ref{rh2},
further implies that
\begin{align*}
\omega_{\lambda}(b;B(x_0,r_0))\|\mathbf{1}_{F}\|_{X}
&\ls\frac{\|\mathbf{1}_{F}\|_{X}}{|B(x_0,r_0)|}
\lf\|[b,T_\Omega](\mathbf{1}_F)\mathbf{1}_{E\setminus Q}\r\|_{X}
\lf\|\mathbf{1}_{E\setminus Q} \r\|_{X'} \\
&\ls\frac{\|\mathbf{1}_{4k_0B(x_0,r_0)}
\|_{X}\|\mathbf{1}_{B(x_0,r_0)}\|_{X'}}{|B(x_0,r_0)|}
\lf\|[b,T_\Omega](\mathbf{1}_F)\mathbf{1}_{E\setminus Q}\r\|_{X}\\
&\ls\lf\|[b,T_\Omega](\mathbf{1}_F)\mathbf{1}_{E\setminus Q}\r\|_{X}.
\end{align*}
This finishes the proof of Proposition \ref{thmm2}.
\end{proof}

To establish the upper estimates of commutators,
we need the following  equivalent $\BMO$-norm characterization
on ball Banach function spaces, namely, Lemma \ref{BMO-X} below,
which is just \cite[Theorems 1.2]{ins} and
an essential tool needed in this article.
\begin{lemma}\label{BMO-X}
Let X be a ball Banach function space such that $M$ is bounded on $X'$ and,
for any $b\in L^1_\loc(\rn)$,
$$\|b\|_{\BMO_X}:=\sup_B\frac1{\|\mathbf1_B\|_X}
\big\| |b-b_B|\mathbf1_B\big\|_X,$$
where the supremum is taken over all balls $B\subset\rn$.
Then there exists a positive constant $C$ such that,
for any $b\in\BMO(\rn)$,
$$C^{-1}\|b\|_{\BMO(\rn)}\le \|b\|_{\BMO_X}\le C\|b\|_{\BMO(\rn)}. $$
\end{lemma}

Next, we give upper estimates of commutators on $X$ as follows.
\begin{proposition}\label{thmm3}
Let $b\in \BMO(\rn)$
and $X$ be a ball Banach function space.
Assume that $\cm$ is bounded on $X$ and $X'$,
and $\Omega\in L^\infty(\mathbb{S}^{n-1})$ satisfies that
there exists an open set $\Lambda\subset \mathbb{S}^{n-1}$ such that
$\Omega$ does not change sign on $\Lambda$.
Let $B(x_0,r_0)$ and $F\subset B(x_1,r_0)$ be as in Lemma \ref{thmb3}.
Then there exists positive constants $C$, $d_0$, and $\delta$
such that, for any $d\in(0,\fz)$ with $d\geq d_0$,
$$
\lf\|[b,T_\Omega](\mathbf{1}_{F})
\mathbf{1}_{B(x_0,2^{d+1}r_0)\setminus B(x_0,2^{d}r_0)}\r\|_{X}
\le C2^{-\delta dn}d \|b\|_{\BMO(\rn)}\|\mathbf{1}_{F}\|_{X},
$$
where the positive constant $C$ is independent of $d$ as well as $B(x_0,r_0)$,
and $d_0$ is a large constant depending only on $k_0$ in Lemma \ref{thmb3}.
\end{proposition}
\begin{proof}
Without loss of generality, we may assume that $\|b\|_{\BMO(\rn)}=1$.
Let
$b\in \BMO(\rn)$ and $B:=B(x_0,r_0)$ with $x_0\in\rn$ and $r_0\in(0,\infty)$;
let
$\varepsilon_0$, $k_0$, $G$, $E$, and $F\subset B(x_1,r_0)$
be as in Lemma \ref{thmb3};
and let
$d_0$ be a positive constant such that $2^{d_0}\in(4k_0,\infty)$.
Then, for any given positive constant $d\ge d_0$
and for any $x\in B(x_0,2^{d+1}r_0)\setminus B(x_0,2^{d}r_0)$ and $y\in B(x_1,r_0)$,
we have
$|x-y|\sim 2^d r_0.$
By this and Lemma \ref{LeHolder}, we conclude that,
for any $x\in B(x_0,2^{d+1}r_0)\setminus B(x_0,2^{d}r_0)$,
\begin{align}\label{upp1}
\lf|[b,T_\Omega](\mathbf{1}_{F})(x)\r|
&=\lf|\int_{\rn}[b(x)-b(y)]\frac{\Omega(x-y)}{|x-y|^n}\mathbf{1}_{F}(y)\,dy \r| \\
&\le \int_{\rn}\lf|b(x)-b_{B(x_1,r_0)}+b_{B(x_1,r_0)}-b(y)\r|
\frac{|\Omega(x-y)|}{|x-y|^n}\mathbf{1}_{F}(y)\,dy\noz\\
&\le \lf|b(x)-b_{B(x_1,r_0)}\r|\int_{\rn}\frac{|\Omega(x-y)|}{|x-y|^n}\mathbf{1}_{F}(y)\,dy\noz\\
&\quad+\int_{\rn}\lf|b_{B(x_1,r_0)}-b(y)\r|
\frac{|\Omega(x-y)|}{|x-y|^n}\mathbf{1}_{F}(y)\,dy. \noz\\
&\ls\frac{\|\Omega\|_{L^\fz(\mathbb{S}^{n-1})}}{2^{dn}r_0^n}
\lf|b(x)-b_{B(x_1,r_0)}\r|
\lf\|\mathbf{1}_{F}\r\|_{X}
\lf\|\mathbf{1}_{F}\r\|_{X'} \noz\\
&\quad+\frac{\|\Omega\|_{L^\fz(\mathbb{S}^{n-1})}}{2^{dn}r_0^n}
\lf\|\mathbf{1}_{F}\r\|_{X}
\lf\||b-b_{B(x_1,r_0)}|\mathbf{1}_{F}\r\|_{X'} \noz\\
&=:{\rm H}_1(x)+{\rm H}_2(x). \noz
\end{align}

Let $v_0\in\{2,3,4,\dots\}$, depending only on $k_0$,
such that $B(x_1,2^{v_0}r_0)\ni x_0$.
Thus, for any $y\in B(x_0,2^{d+1}r_0)$, we have
$$|y-x_1|\le|y-x_0|+|x_0-x_1|
\le2^{d+1}r_0+2^{v_0}r_0
\le2^{\max\{d+1,v_0\}+1}r_0
\le2^{d+v_0}r_0,$$
which implies that
\begin{align}\label{Bx0x1}
B(x_0,2^{d+1}r_0)\subset  B(x_1,2^{d+v_0}r_0).
\end{align}
Moreover, by $\|b\|_{\BMO(\rn)}=1$, it is easy to see that
$$\lf|b_{B(x_1,r_0)}-b_{2^{d+v_0}B(x_1,r_0)} \r|\le(d+v_0)2^n.$$
From this, \eqref{upp1}, \eqref{Bx0x1}, Lemmas \ref{BMO-X},
Lemma \ref{as1} with $\tz$ replaced by $\eta\in(1,\fz)$ in Remark \ref{eta},
the conclusion $F\subset B(x_1,r_0)$ in Lemma \ref{thmb3},
and Lemma \ref{rh2} with $B$ replaced by $B(x_1,r_0)$,
we deduce that
\begin{align}\label{H1}
&\lf\|{\rm H}_1\mathbf{1}_{B(x_0,2^{d+1}r_0)\setminus B(x_0,2^{d}r_0)}\r\|_X\\
&\quad\ls2^{-dn}r_0^{-n}\lf\|\mathbf{1}_{F}\r\|_{X}
\lf\|\mathbf{1}_{F}\r\|_{X'}
\lf\||b-b_{B(x_1,r_0)}|\mathbf{1}_{B(x_0,2^{d+1}r_0)\setminus B(x_0,2^{d}r_0)} \r\|_X\noz\\
&\quad\ls2^{-dn}r_0^{-n}\lf\|\mathbf{1}_{F}\r\|_{X}
\lf\|\mathbf{1}_{F}\r\|_{X'}
\lf\||b-b_{B(x_1,r_0)}|\mathbf{1}_{B(x_1,2^{d+v}r_0)} \r\|_X\noz\\
&\quad\ls2^{-dn}r_0^{-n}\lf\|\mathbf{1}_{F}\r\|_{X}
\lf\|\mathbf{1}_{F}\r\|_{X'}
\lf[\lf\|\lf|b-b_{B(x_1,2^{d+v_0}r_0)}\r|\mathbf{1}_{B(x_1,2^{d+v_0}r_0)} \r\|_X
+d\lf\|\mathbf{1}_{B(x_1,2^{d+v_0}r_0)}\r\|_X\r]\noz\\
&\quad\ls2^{-dn}r_0^{-n}d\lf\|\mathbf{1}_{F}\r\|_{X}
\lf\|\mathbf{1}_{F}\r\|_{X'}
\lf\|\mathbf{1}_{B(x_1,2^{d+v_0}r_0)}\r\|_X\noz\\
&\quad\ls2^{-(1-1/\eta) dn}d
\lf\|\mathbf{1}_{F}\r\|_{X}
\lf\|\mathbf{1}_{B(x_1,r_0)}\r\|_{X'}
\lf\|\mathbf{1}_{B(x_1,r_0)}\r\|_X
r_0^{-n}\noz\\
&\quad\ls2^{-(1-1/\eta) dn}d
\lf\|\mathbf{1}_{F}\r\|_{X}.\noz
\end{align}
Similarly, by \eqref{upp1}, the fact $F\subset B(x_1,r_0)$ again,
Lemmas \ref{Lesdual}, \ref{BMO-X}, \ref{as1}, and \ref{rh2},
we conclude that
\begin{align}\label{H2}
&\lf\|{\rm H}_2\mathbf{1}_{B(x_0,2^{d+1}r_0)\setminus B(x_0,2^{d}r_0)}\r\|_X \\
&\quad\ls 2^{-dn}r_0^{-n}\lf\|\mathbf{1}_{F}\r\|_{X}
\lf\||b-b_{B(x_1,r_0)}|\mathbf{1}_{F}\r\|_{X'}
\lf\|\mathbf{1}_{B(x_0,2^{d+1}r_0)\setminus B(x_0,2^{d}r_0)}\r\|_X \noz \\
&\quad\ls 2^{-dn}r_0^{-n}\lf\|\mathbf{1}_{F}\r\|_{X}
\lf\||b-b_{B(x_1,r_0)}|\mathbf{1}_{B(x_1,r_0)}\r\|_{X'}
\lf\|\mathbf{1}_{B(x_1,2^{d}r_1)}\r\|_X \noz\\
&\quad\ls 2^{-(1-1/\eta) dn}d
\lf\|\mathbf{1}_{F}\r\|_{X}
\lf\|\mathbf{1}_{B(x_1,r_0)}\r\|_{X'}
\lf\|\mathbf{1}_{B(x_1,r_0)}\r\|_X
r_0^{-n}\noz\\
&\quad\ls2^{-(1-1/\eta) dn}d
\lf\|\mathbf{1}_{F}\r\|_{X}.\noz
\end{align}
Combining \eqref{upp1}, \eqref{H1}, and \eqref{H2},
and letting $\dz:=1-1/\eta$,
we then complete the proof of Proposition \ref{thmm3}.
\end{proof}

\begin{proof}[Proof of Theorem \ref{thm-cpt2}]
By Theorem \ref{thm-bdd-2}, we conclude that $b\in\BMO(\rn)$ and then,
without loss of generality, we may assume that $\|b\|_{{\rm BMO}(\rn)}=1$.
To show $b\in{\rm CMO}(\rn)$,
we use a contradiction argument via Lemma \ref{l-cmo char}.
Now, observe that, if $b\notin {\rm CMO}(\rn)$, then $b$ does not satisfy at least one of
(i), (ii), and (iii) of Lemma \ref{l-cmo char}.
To finish the proof of this theorem,
we only need to show that, if $b$ does not satisfy
one of (i), (ii), and (iii) of Lemma \ref{l-cmo char},
then $[b, T_\Omega]$ is not compact on $X$.
We prove this by three cases as follows.

{\bf Case i)} Suppose that $b$ does not satisfy Lemma \ref{l-cmo char}(i).
In this case, there exist a constant $\dz_0\in(0, 1)$ and a sequence of balls
$\{ B_j\}_{j\in\nn}$, with $| B_j|\to0$ as $j\to\fz$,
such that, for any $j\in\nn$,
\begin{align}\label{upp2}
\oz_\lz(b; B_j)\ge\dz_0,
\end{align}
where $\lz\in(0,1/2]$ and $\oz_\lz(b; B_j)$ is as in \eqref{z3}
with $f$ and $B$ replaced, respectively, by $b$ and $B_j$.
For any given ball $B:=B(x_0,r_0)$,
let $E$ and $F$ be the sets associated with $B$ in Lemma \ref{thmb3},
$$f:=\|\mathbf1_{F}\|_X^{-1}\mathbf1_{F},$$
and $2C_0:=C_{(\lz,k_0,\varepsilon_0,n)}$ in Proposition \ref{thmm2}.
Then, by Proposition \ref{thmm2}, we conclude that, for any measurable set
$Q\subset\rn$ with $|Q|\le\frac{\lz}{8}|B|$,
\begin{align}\label{low1}
\lf\|[b,T_\Omega](f)
\mathbf{1}_{E\setminus Q}\r\|_{X}
\ge2C_0\oz_{\lz}(b;B).
\end{align}
For such chosen $C_0$ and $\dz_0$, by Proposition \ref{thmm3},
there exists a positive constant $d_0$ such that
\begin{align}\label{upp3}
\lf\|[b,T_\Omega](f)
\mathbf{1}_{\rn\setminus B(x_0,2^{d_0}r_0)}\r\|_{X}
\le\sum_{k=0}^\fz \lf\|[b,T_\Omega](\mathbf{1}_{F})
\mathbf{1}_{B(x_0,2^{d_0+k+1}r_0)\setminus B(x_0,2^{d_0+k}r_0)}\r\|_{X}
\le C_0\dz_0.
\end{align}
Take a subsequence of $\{ B_j\}_{j\in\nn}$,
still denoted by $\{ B_j\}_{j\in\nn}$,
such that, for any $j\in\nn$,
$$\frac{|{B}_{j+1}|}{| B_j|}
\le\min\lf\{\frac{\lz^2}{64},2^{-2 d_0 n}\r\}.$$
Let $\widetilde{B}_j:=(|{B}_{j-1}|/|{B}_{j}|)^{1/2n}  B_j$
for any $j\in\nn$ and $j\ge2$.
Then it is easy to see that, for any $j\in\nn$ and $j\ge2$,
$$\lf(\frac{|{B}_{j-1}|}{| B_j|}\r)^{\frac1{2n}}\ge2^{d_0}
\quad{\rm and}\quad |\widetilde{B}_j|\le\frac{\lz}{8}|{B}_{j-1}|.$$
From this and the monotonicity of $\{B_j\}_{j\in\nn}$,
we deduce that, for any integers $k$ and $j$ with $k>j\ge2$,
\begin{align}\label{Bkj}
2^{d_0}B_k\subset \widetilde{B}_k
\quad{\rm and}\quad
|\widetilde{B}_k|\le\frac{\lz}{8}|{B}_{k-1}|\le\frac{\lz}{8}|{B}_{j}|
\end{align}
Now, for any $j\in\nn$, let $E_j$ and $F_j$ be the sets associated with $B_j$
as in Lemma \ref{thmb3} with $B$ replaced by $B_j$, and
$$f_j:=\lf\|\mathbf1_{F_j}\r\|_X^{-1}\mathbf1_{F_j}.$$
Then, for any integers $k$ and $j$ with $k>j\ge2$,
by \eqref{low1}, \eqref{upp2}, \eqref{Bkj}, and \eqref{upp3}, we conclude that
$$\lf\|[b,T_\Omega](f_j)
\mathbf{1}_{E_j\setminus \widetilde{B}_k}\r\|_{X}
\ge2C_0\oz_{\lz}(b;B)
\ge2C_0\dz_0$$
and
$$\lf\|[b,T_\Omega](f_k)
\mathbf{1}_{E_j\setminus \widetilde{B}_k}\r\|_{X}
\le\lf\|[b,T_\Omega](f_k)
\mathbf{1}_{\rn\setminus 2^{d_0}B_k}\r\|_{X}
\le C_0\dz_0,$$
which further implies that
\begin{align*}
&\lf\|[b,T_\Omega](f_j)-[b,T_\Omega](f_k)\r\|_{X}\\
&\quad\ge\lf\|\lf\{[b,T_\Omega](f_j)-[b,T_\Omega](f_k)\r\}
\mathbf{1}_{E_j\setminus \widetilde{B}_k}\r\|_{X}\\
&\quad\ge\lf\|[b,T_\Omega](f_j)
\mathbf{1}_{E_j\setminus \widetilde{B}_k}\r\|_{X}
-\lf\|[b,T_\Omega](f_k)
\mathbf{1}_{E_j\setminus \widetilde{B}_k}\r\|_{X}
\ge C_0\dz_0.
\end{align*}
Therefore, $\{[b,T_\Omega]f_j\}_{j\in\nn}$ is not relatively compact in $X$,
which leads to a contradiction with  the compactness of $[b,T_\Omega]$ on $X$.
This shows that $b$ satisfies Lemma \ref{l-cmo char}(i),
which is the desired conclusion.

{\bf Case ii)} Suppose that $b$ dose not satisfy Lemma \ref{l-cmo char}(ii).
In this case, similarly to above Case i),
there exist a $\dz_0\in(0, 1)$ and a sequence of balls
$\{B_j\}_{j\in\nn}$ such that, for any $j\in\nn$,
$$\oz_\lz(b; B_j)\ge\dz_0
\quad{\rm and}\quad
\frac{|{B}_{j}|}{| B_{j+1}|}
\le\min\lf\{\frac{\lz^2}{64},2^{-2 d_0 n}\r\},$$
where $C_0$ and $d_0$ are as in Case i) such that \eqref{low1}
and \eqref{upp3} hold true.
For any $j\in\nn$, let $E_j$, $F_j$, $f_j$ be as in Case i),
and $\widetilde{B}_j:=(|{B}_{j}|/|{B}_{j-1}|)^{1/2n}  B_{j-1}$
for any $j\ge2$.
Then it is easy to see that, for any integers $k$ and $j$ with $2\le k\le j$,
$$2^{d_0}B_{k-1}\subset \widetilde{B}_k
\quad{\rm and}\quad
|\widetilde{B}_k|\le\frac{\lz}{8}|{B}_{j}|$$
Using a method similar to that used in Case i),
we conclude that
$$\lf\|[b,T_\Omega](f_j)-[b,T_\Omega](f_k)\r\|_{X}\ge C_0\dz_0,$$
and hence
$\{[b,T_\Omega]f_j\}_{j\in\nn}$ is not relatively compact in $X$,
which is a contradiction.
This shows that $b$ satisfies Lemma \ref{l-cmo char}(ii),
which is also the desired conclusion.

{\bf Case iii)} Suppose that $b$ does not satisfy Lemma \ref{l-cmo char}(iii).
In this case, there exist a $\dz_0\in(0, 1)$ and a sequence of balls
$\{B_j\}_{j\in\nn}$ such that, for any $j\in\nn$,
\begin{align}\label{upp4}
\oz_\lz(b; B_j)\ge\dz_0.
\end{align}
From this and Cases i) and ii), we deduce that there exist
a constant $d_1\in[d_0,\fz)$ with $d_0$ as in Lemma \ref{thmb3},
and a subsequence of balls $\{B_j\}_{j\in\nn}$,
still denoted by $\{B_j\}_{j\in\nn}$,
such that
$$|B_j|\sim1,\quad\forall\,j\in\nn$$
and
$$2^{d_1} B_i\cap 2^{d_1} B_j=\emptyset,\quad\forall\,i\neq j.$$
For any $j\in\nn$, let $E_j$, $F_j$, $f_j$, and $C_0$ be as in Case i).
Notice that, for any positive integers $k$ and $j$,
$$\lf(2^{d_0} B_k\cap E_j\r)\subset\lf(2^{d_1} B_k\cap 2^{d_1} B_j\r)=\emptyset.$$
By this, Proposition \ref{thmm2} with $Q:=\emptyset$,
and \eqref{upp4},
we conclude that, for any positive integers $k$ and $j$,
\begin{align}\label{low2}
\lf\|[b,T_\Omega](f_j)
\mathbf{1}_{E_j\setminus2^{d_0} B_k}\r\|_{X}
=\lf\|[b,T_\Omega](f_j)
\mathbf{1}_{E_j}\r\|_{X}
\ge2C_0\oz_{\lz}(b;B)
\ge2C_0\tz_0.
\end{align}
Moreover, from Proposition \ref{thmm3}, we deduce that,
for any positive integers $k$ and $j$,
\begin{align}\label{upp5}
\lf\|[b,T_\Omega](f_k)
\mathbf{1}_{E_j\setminus2^{d_0} B_k}\r\|_{X}
\le\lf\|[b,T_\Omega](f_k)
\mathbf{1}_{\rn\setminus 2^{d_0}B_k}\r\|_{X}
\le C_0\dz_0.
\end{align}
Combining \eqref{low2} and \eqref{upp5}, we obtain
\begin{align*}
&\lf\|[b,T_\Omega](f_j)-[b,T_\Omega](f_k)\r\|_{X}\\
&\quad\ge\lf\|\lf\{[b,T_\Omega](f_j)-[b,T_\Omega](f_k)\r\}
\mathbf{1}_{E_j\setminus2^{d_0} B_k}\r\|_{X}\\
&\quad\ge\lf\|[b,T_\Omega](f_j)
\mathbf{1}_{E_j\setminus2^{d_0} B_k}\r\|_{X}
-\lf\|[b,T_\Omega](f_k)
\mathbf{1}_{E_j\setminus2^{d_0} B_k}\r\|_{X}
\ge C_0\dz_0
\end{align*}
and hence
$\{[b,T_\Omega]f_j\}_{j\in\nn}$ is not relatively compact in $X$,
which is a contradiction.
This shows that $b$ satisfies Lemma \ref{l-cmo char}(iii),
which completes the proof of Theorem \ref{thm-cpt2}.
\end{proof}

\section{Applications\label{s5}}

In this section, we apply
Theorems \ref{thm-bdd-1}, \ref{thm-bdd-2}, \ref{thm-cpt}, and \ref{thm-cpt2},
respectively, to six concrete examples of ball Banach function spaces,
namely, Morrey spaces (see Subsection \ref{s5.1} below),
mixed-norm Lebesgue spaces (see Subsection \ref{s5.2} below),
variable Lebesgue spaces (see Subsection \ref{s5.3} below), weighted Lebesgue spaces
(see Subsection \ref{s5.4} below),
Orlicz spaces (see Subsection \ref{s5.5x} below),
and Orlicz-slice spaces (see Subsection \ref{s5.5} below). Observe that, among
these six examples, only variable Lebesgue spaces and Orlicz spaces are
Banach function spaces as in Remark \ref{ball-bounded}(ii),
while the other four examples are ball Banach function spaces, which are not necessary to be
Banach function spaces.

\subsection{Morrey spaces\label{s5.1}}

Recall that, due to the applications in elliptic partial differential equations,
the Morrey space $M_r^p(\rn)$ with $0<r\le p<\infty$
was introduced by Morrey \cite{Mo} in 1938.
In recent decades,
there exists an increasing interest in applications
of Morrey spaces to various areas of analysis
such as partial differential equations, potential theory, and
harmonic analysis; see, for instance, \cite{a15,a04,cf,JW,sdh20,sdh20II,tyy19,ysy}.

\begin{definition}
Let $0<r\le p<\infty$.
The \emph{Morrey space $M_r^p(\rn)$} is defined
to be the set of all measurable functions $f$ on $\rn$ such that
$$
\|f\|_{M_r^p(\rn)}:=\sup_{B\in\BB}|B|^{1/p-1/r}\|f\|_{L^p(B)}<\infty,
$$
where $\BB$ is as in \eqref{Eqball} (the set of all balls of $\rn$).
\end{definition}

\begin{remark}
Observe that, as was pointed out in \cite[p.\,86]{SHYY}, $M_r^p(\rn)$
may not be a Banach function space, but it is a ball
Banach function space as in Definition \ref{Debqfs}.
\end{remark}

Let $p\in(1,\fz)$ and $r\in(0,p]$. From \cite[Theorem 1]{cf},
it follows that the
Hardy--Littlewood maximal operator $\cm$ is bounded on $M_r^p(\rn)$.
Recall that the associate space of the Morrey space is the block space
(see, for instance, \cite[Theorem 4.1]{ST}) and
$\cm$ is bounded on block spaces
(see, for instance, \cite[Theorem 3.1]{ch14} and \cite[Lemma 5.7]{H15}).
Using these and Definition \ref{Debf}, we can easily show that, for any given $s\in(0,p)$,
$\cm$ is bounded on $(X)'$ and $(X^{1/s})'$, where $X:=M_r^p(\rn)$.
Thus, all the assumptions of the main theorems in Sections \ref{s2} and \ref{s3} are satisfied.
Using Theorems \ref{thm-bdd-1},
\ref{thm-cpt}, and \ref{thm-cpt2},
we obtain
the following characterization of the boundedness and the compactness of commutators on Morrey spaces,
respectively, via ${\rm BMO}(\rn)$ and ${\rm CMO}(\rn)$.

\begin{theorem}\label{bdd-Morrey}
Let $p\in(1,\fz)$ and $r\in(0,p]$.
Then Theorems \ref{thm-bdd-1}, \ref{thm-bdd-2},
\ref{thm-cpt}, and \ref{thm-cpt2} hold true with $X$ replaced by $M_r^p(\rn)$.
\end{theorem}

\begin{remark}\label{rem-bdd-Morrey}
\begin{itemize}
\item [{\rm(i)}]
The boundedness of commutators on Morrey spaces was first obtained by
Di Fazio and Ragusa \cite[Theorem 1]{dr91}.
Indeed, Di Fazio and Ragusa \cite{dr91} proved Theorem \ref{bdd-Morrey} under the assumption
that $\Omega\in{\rm Lip}(\mathbb{S}^{n-1})$ satisfies \eqref{deg0} and \eqref{mean0},
which is a spacial cases of Theorem \ref{bdd-Morrey}.

\item [{\rm(ii)}]
Let $p\in(1,\fz)$ and $r\in(0,p]$.
Theorem \ref{thm-cpt} with $X$ replaced by $M_r^p(\rn)$
was obtained by Chen et al. \cite[Theorem 1.1]{CDW12}.
On the other hand,
Chen et al. \cite[Theorem 1.2]{CDW12} showed the necessity
under the assumption that $\Omega\in{\rm Lip}(\mathbb{S}^{n-1})$
satisfies \eqref{deg0} and \eqref{mean0},
which is stronger than Theorem \ref{bdd-Morrey}.
\end{itemize}
\end{remark}

\subsection{Mixed-norm Lebesgue spaces\label{s5.2}}

The mixed-norm Lebesgue space $L^{\vec{p}}(\rn)$
was studied by Benedek and Panzone \cite{BP} in 1961, which can be traced
back to H\"ormander \cite{H1}. Later on, in 1970, Lizorkin \cite{l70} further developed both the theory of
multipliers of Fourier integrals and estimates of convolutions
in the mixed-norm Lebesgue spaces.
Particularly,
in order to meet the requirements arising in the study of the boundedness of operators,
partial differential equations, and some other analysis
subjects, the real-variable theory of mixed-norm
function spaces, including mixed-norm Morrey spaces,
mixed-norm Hardy spaces, mixed-norm Besov spaces,
and mixed-norm Triebel--Lizorkin spaces, has rapidly been developed
in recent years (see, for instance,
\cite{cgn17bs,gjn17,tn,noss20,hy,HLYY,hlyy19}).

\begin{definition}\label{mix}
Let $\vec{p}:=(p_1,\ldots,p_n)\in(0,\infty]^n$.
The \emph{mixed-norm Lebesgue space $L^{\vec{p}}(\rn)$} is defined
to be the set of all measurable functions $f$ on $\rn$ such that
$$
\|f\|_{L^{\vec{p}}(\rn)}:=\lf\{\int_{\rr}\cdots\lf[\int_{\rr}|f(x_1,\ldots,x_n)|^{p_1}\,dx_1\r]
^{\frac{p_2}{p_1}}\cdots\,dx_n\r\}^{\frac{1}{p_n}}<\infty
$$
with the usual modifications made when $p_i=\infty$ for some $i\in\{1,\ldots,n\}$.
\end{definition}

In this subsection, for any $\vec{p}:=(p_1,\ldots,p_n)\in(0,\infty)^n$, we always let
$p_-:= \min\{p_1, \ldots , p_n\}$ and  $p_+ := \max\{p_1, \ldots , p_n\}$.

Let $\vec{p}\in(1,\infty)^n$. Then $\cm$ is bounded on $L^{\vec{p}}(\rn)$
(see, for instance, \cite[Lemma 3.5]{HLYY}).
Applying this and the dual theorem (see, for instance, \cite[Theorem 1.a]{BP}),
we can easily show that, for any given $s\in(0,p_-)$,
$\cm$ is bounded on $X'$ and $(X^{1/s})'$, where $X:=L^{\vec{p}}(\rn)$.
Thus, all the assumptions of the main theorems in Sections \ref{s2} and \ref{s3} are satisfied.
Using Theorems \ref{thm-bdd-1},
\ref{thm-cpt}, and \ref{thm-cpt2},
we obtain
the following characterization of the boundedness
and the compactness of commutators on mixed-norm Lebesgue spaces.

\begin{theorem}\label{bdd-mixLp}
Let $\vec{p}:=(p_1,\ldots,p_n)\in(1,\infty)^n$.
Then Theorems \ref{thm-bdd-1}, \ref{thm-bdd-2},
\ref{thm-cpt}, and \ref{thm-cpt2} hold true with $X$ replaced by $L^{\vec{p}}(\rn)$.
\end{theorem}

\begin{remark}\label{rem-mixLp}
To the best of our knowledge, the results of
Theorem \ref{bdd-mixLp} are totally new.
\end{remark}

\subsection{Variable Lebesgue spaces\label{s5.3}}

Let $p(\cdot):\ \rn\to[0,\infty)$ be a measurable function. Then the \emph{variable Lebesgue space
$L^{p(\cdot)}(\rn)$} is defined to be the set of all measurable functions $f$ on $\rn$ such that
$$
\|f\|_{L^{p(\cdot)}(\rn)}:=\inf\lf\{\lambda\in(0,\infty):\ \int_\rn\lf[\frac{|f(x)|}{\lambda}\r]^{p(x)}\,dx\le1\r\}<\infty.
$$
We refer the reader to \cite{N1,N2,KR,CUF,DHR} for more details on variable Lebesgue spaces.

For any measurable function $p(\cdot):\ \rn\to(0,\infty)$, in this subsection, we let
$$
 \widetilde{p}_-:=\underset{x\in\rn}{\essinf}\,p(x)\quad\text{and}\quad
 \widetilde p_+:=\underset{x\in\rn}{\esssup}\,p(x).
$$
If $1<\widetilde p_-\le \widetilde p_+<\infty$, then, similarly to the proof of \cite[Theorem 3.2.13]{DHHR},
we know that $L^{p(\cdot)}(\rn)$ is
a Banach function space and hence a ball Banach function space.

A measurable function $p(\cdot):\ \rn\to(0,\infty)$ is said to be
\emph{globally log-H\"older continuous} if there exist a $p_{\infty}\in\rr$
and a positive constant $C$ such that, for any
$x,\ y\in\rn$,
$$
|p(x)-p(y)|\le C\frac{1}{\log(e+1/|x-y|)}
$$
and
$$
|p(x)-p_\infty|\le C\frac{1}{\log(e+|x|)}.
$$

Let $p(\cdot):\ \rn\to(0,\infty)$ be a globally
log-H\"older continuous function satisfying
$1<\widetilde p_-\le \widetilde p_+<\infty$.
Adamowicz et al. \cite[Theorem 1.7]{ahh} obtained the
boundedness of the Hardy--Littlewood maximal operator on variable Lebesgue spaces;
see also \cite{CUF,CUW}.
Furthermore, from this and the dual theorem
(see, for instance, \cite[Theorem 2.80]{CUF}),
we deduce that, for any given $s\in(0,\widetilde p_-)$ and
$q\in(\widetilde p_+,\infty]$, $M$ is bounded on $X'$ and $(X^{1/s})'$,
where $X:=L^{p(\cdot)}(\rn)$.
Thus, all the assumptions of the main theorems in Sections \ref{s2} and \ref{s3} are satisfied.
Using Theorems \ref{thm-bdd-1},
\ref{thm-cpt}, and \ref{thm-cpt2},
we obtain the following characterization of the boundedness
and the compactness of commutators on variable Lebesgue spaces.

\begin{theorem}\label{bdd-varLp}
Let $p(\cdot):\ \rn\to(0,\infty)$ be a globally
log-H\"older continuous function satisfying
$1<\widetilde p_-\leq \widetilde p_+<\infty$.
Then Theorems \ref{thm-bdd-1}, \ref{thm-bdd-2},
\ref{thm-cpt}, and \ref{thm-cpt2} hold true with $X$ replaced by $L^{p(\cdot)}(\rn)$.
\end{theorem}

\begin{remark}\label{rem-varLp}
The boundedness characterization of commutators on variable Lebesgue spaces
was first studied by Karlovich and Lerner \cite[Theorem 1.1]{KL05};
meanwhile, they point out in \cite[Remark 4.3]{KL05} that
the corresponding conclusion also holds true in Banach function spaces.
Moreover, Guo et al. \cite[Theorem 2.1]{glw} proved a generalization
for the necessity part in ball Banach function spaces, based on a weaker
assumption than \cite[Theorem 1.1(b)]{KL05}.
As for the compactness characterization, to the best of our knowledge,
the corresponding results of Theorem \ref{bdd-varLp} is totally new.
\end{remark}

\subsection{Weighted Lebesgue spaces\label{s5.4}}

It is worth pointing out that a weighted Lebesgue space with an $A_\infty(\rn)$-weight may
not be a Banach function space; see \cite[Section 7.1]{SHYY}.
From \cite[Theorem 3.1(b)]{AJ}, it follows that, for any $p\in(1,\fz)$,
\begin{align}\label{Mbdd-Ap}
\cm {\rm\ \ is\ bounded\ on\ } L^p_\oz(\rn) {\rm\ if\ and\ only\ if\ } \oz\in A_p(\rn).
\end{align}
Therefore, $L^p_\omega(\rn)$ satisfies Assumption \ref{as1}
for any given $p\in(1,\fz)$ and $\oz\in A_p(\rn)$.
Moreover, from \cite[Theorem 2.7.4]{DHHR},
we deduce that, when $p\in(1,\infty)$ and $\omega\in A_\infty(\rn)$,
$$\lf[L_\omega^p(\rn)\r]'=L^{p'}_{\omega^{1-p'}}(\rn),$$
where $[L_\omega^p(\rn)]'$ denotes the associated space of $L_\omega^p(\rn)$
as in \eqref{asso} with $X:=L^p_\omega(\rn)$.
By this, \eqref{Mbdd-Ap}, and the observation that
$$\oz\in A_p(\rn){\rm\ if\ and\ only\ if\ } \oz^{1-p'}\in A_{p'}(\rn),$$
we conclude that $\cm$ is bounded on $[L_\omega^p(\rn)]'$.
Furthermore, since $A_p(\rn)=\bigcup_{q\in(1,p)}A_q(\rn)$
(see, for instance, \cite[p.\,139, Corollary 7.9(a)]{d}),
it follows that, for any $\oz\in A_p(\rn)$,
there exists some $s\in(1,\fz)$ such that
$\oz\in A_{p/s}(\rn)$, and hence $\oz^{1-(p/s)'}\in A_{(p/s)'}(\rn)$.
By this, \eqref{Mbdd-Ap}, and the fact that
$$\lf(\lf[L_\omega^p(\rn)\r]^{1/s}\r)'
=\lf[L_\omega^{p/s}(\rn)\r]'=L^{(p/s)'}_{\omega^{1-(p/s)'}}(\rn),$$
we conclude that $\cm$ is bounded on $([L^p_\oz(\rn)]^{1/s})'$.
Thus, all the assumptions of the main theorems in Sections \ref{s2} and \ref{s3} are satisfied.
Using Theorems \ref{thm-bdd-1},
\ref{thm-cpt}, and \ref{thm-cpt2},
we immediately obtain the following characterization of the boundedness
and the compactness of commutators on weighted Lebesgue spaces.

\begin{theorem}\label{bdd-wtdLp}
Let $p\in(1,\infty)$ and $\omega\in A_p(\rn)$.
Then Theorems \ref{thm-bdd-1}, \ref{thm-bdd-2},
\ref{thm-cpt}, and \ref{thm-cpt2} hold true with $X$ replaced by $L_\omega^p(\rn)$.
\end{theorem}

\begin{remark}\label{rem-wtdLp}
The boundedness characterization of Theorem \ref{bdd-wtdLp}
was obtained in \cite[p.\,129, Theorem 2.4.3]{ldy},
and the compactness characterization of Theorem \ref{bdd-wtdLp}
was obtained in \cite[Theorems 1.4 and 1.5]{gwy20}.
\end{remark}

\subsection{Orlicz spaces\label{s5.5x}}

Birnbaum and Orlicz \cite{bo31} (see also Orlicz \cite{o32}) introduced
the Orlicz space which is another generalization of $L^p(\rn)$.
Since then,  Orlicz spaces have been well developed and widely used in
harmonic analysis, partial differential equations, potential theory,
probability, and some other fields of analysis;
see, for instance, \cite{aikm00,mw08,RR} and their references.

First, we recall the notions of both Orlicz functions and Orlicz spaces.

\begin{definition}\label{or}
A function $\Phi:\ [0,\infty)\ \to\ [0,\infty)$ is called an \emph{Orlicz function} if it is
non-decreasing and satisfies $\Phi(0)= 0$, $\Phi(t)>0$ whenever $t\in(0,\infty)$, and $\lim_{t\to\infty}\Phi(t)=\infty$.
\end{definition}

An Orlicz function $\Phi$ as in Definition \ref{or} is said to be
of \emph{lower} (resp., \emph{upper}) \emph{type} $p$ with
$p\in\rr$ if
there exists a positive constant $C_{(p)}$, depending on $p$, such that, for any $t\in[0,\infty)$
and $s\in(0,1)$ [resp., $s\in [1,\infty)$],
\begin{equation*}
\Phi(st)\le C_{(p)}s^p \Phi(t).
\end{equation*}
A function $\Phi:\ [0,\infty)\ \to\ [0,\infty)$ is said to be of
 \emph{positive lower} (resp., \emph{upper}) \emph{type} if it is of lower
 (resp., upper) type $p$ for some $p\in(0,\infty)$.
\begin{definition}\label{fine}
Let $\Phi$ be an Orlicz function with positive lower type $p_{\Phi}^-$ and positive upper type $p_{\Phi}^+$.
The \emph{Orlicz space $L^\Phi(\rn)$} is defined
to be the set of all measurable functions $f$ on $\rn$ such that
 $$\|f\|_{L^\Phi(\rn)}:=\inf\lf\{\lambda\in(0,\infty):\ \int_{\rn}\Phi\lf(\frac{|f(x)|}{\lambda}\r)\,dx\le1\r\}<\infty.$$
\end{definition}

It is well known that, if $p_{\Phi}^-$, $p_{\Phi}^+\in(1,\fz)$,
then the dual space of $L^\Phi(\rn)$ is $L^\Psi(\rn)$,
where $\Psi$ denotes the \emph{complementary function}
defined by setting $\Psi(t):=\sup\{xt-\Phi(x):\ x\in[0,\fz)\}$ for any $t\in[0,\fz)$
(see \cite[Definition 2.14]{ZYYW}) of $\Phi$.
Moreover, $L^\Phi(\rn)$ is
a Banach function space and hence a ball Banach function space.
Furthermore, $\cm$ is bounded on $L^\Phi(\rn)$, $L^\Psi(\rn)$, and
$([L^\Phi(\rn)]^{1/s})'$ for some $s\in(0,1)$.
These basic properties can be found in, for instance, \cite[Subsection 7.6]{SHYY}.
Thus, all the assumptions of the main theorems in Sections \ref{s2} and \ref{s3} are satisfied.
Using Theorems \ref{thm-bdd-1},
\ref{thm-cpt}, and \ref{thm-cpt2},
we immediately obtain the following characterization of the boundedness
and the compactness of commutators on Orlicz spaces.

\begin{theorem}\label{bdd-Orlicz}
Let $p_{\Phi}^-,\ p_{\Phi}^+\in(1,\infty)$ and
$\Phi$ be an Orlicz function with positive lower type $p_{\Phi}^-$
and positive upper type $p_{\Phi}^+$.
Then Theorems \ref{thm-bdd-1}, \ref{thm-bdd-2},
\ref{thm-cpt}, and \ref{thm-cpt2} hold true with $X$ replaced by $L^\Phi(\rn)$.
\end{theorem}

\begin{remark}\label{rem-Orlicz}
To the best of our knowledge,
the results of Theorem \ref{bdd-Orlicz} are totally new.
\end{remark}

\subsection{Orlicz-slice spaces\label{s5.5}}

Now, we recall the notion of Orlicz-slice spaces.
\begin{definition}\label{so}
Let $t,\ r\in(0,\infty)$ and $\Phi$ be an Orlicz function with positive lower type $p_{\Phi}^-$ and
positive upper type $p_{\Phi}^+$. The \emph{Orlicz-slice space} $(E_\Phi^r)_t(\rn)$
is defined to be the set of all measurable functions $f$ on $\rn$
such that
$$
\|f\|_{(E_\Phi^r)_t(\rn)}
:=\lf\{\int_{\rn}\lf[\frac{\|f\mathbf{1}_{B(x,t)}\|_{L^\Phi(\rn)}}
{\|\mathbf{1}_{B(x,t)}\|_{L^\Phi(\rn)}}\r]^r\,dx\r\}^{\frac{1}{r}}<\infty.
$$
\end{definition}

\begin{remark}
By \cite[Lemma 2.28]{ZYYW}, we know that the Orlicz-slice space $(E_\Phi^r)_t(\rn)$
is a ball Banach function space, but it may not be a Banach function space
(see, for instance, \cite[Remark 7.41(i)]{ZWYY}).
\end{remark}

The Orlicz-slice space was introduced by Zhang et al. \cite{ZYYW},
which is a generalization of the slice spaces proposed by
Auscher and Mourgoglou \cite{AM2014,APA} and Wiener amalgam space
in \cite{h75,knt,h19}.
Let $t\in(0,\fz),\ r\in(1,\infty)$, and $\Phi$ be an Orlicz function
with positive lower type $p_{\Phi}^-\in(1,\fz)$ and
positive upper type $p_{\Phi}^+\in(1,\fz)$.
Then $\cm$ is bounded on $(E_\Phi^r)_t(\rn)$
with the implicit positive constant independent of $t$
(see \cite[Proposition 2.22]{ZYYW}).
Besides, from \cite[Theorem 2.26]{ZYYW}, it follows that
$$\lf[(E_\Phi^r)_t(\rn)\r]'=(E_\Psi^{r'})_t(\rn),$$
where $\Psi$ is the complementary function of $\Phi$.
By this, \cite[Proposition 2.22 and Lemma 4.4]{ZYYW}, we conclude that
$\cm$ is bounded on $[(E_\Phi^r)_t(\rn)]'$ and
$([(E_\Phi^r)_t(\rn)]^{1/s})'$ for some $s\in(0,1)$.
Thus, all the assumptions of the main theorems in Sections \ref{s2} and \ref{s3} are satisfied.
Using Theorems \ref{thm-bdd-1},
\ref{thm-cpt}, and \ref{thm-cpt2},
we immediately obtain the following characterization of the boundedness
and the compactness of commutators on Orlicz-slice spaces.

\begin{theorem}\label{bdd-Slice}
Let $t\in(0,\infty)$, $r,\ p_{\Phi}^-,\ p_{\Phi}^+\in(1,\infty)$,
and $\Phi$ be an Orlicz function with positive lower type $p_{\Phi}^-$
and positive upper type $p_{\Phi}^+$.
Then Theorems \ref{thm-bdd-1}, \ref{thm-bdd-2},
\ref{thm-cpt}, and \ref{thm-cpt2} hold true with $X$ replaced by $(E_\Phi^r)_t(\rn)$.
\end{theorem}

\begin{remark}\label{rem-Slice}
To the best of our knowledge,
the results of Theorem \ref{bdd-Slice} are totally new,
even for the slice spaces in \cite{AM2014,APA}, namely,
$(E_\Phi^r)_t(\rn)$ with $\Phi(u):=u^p$ for any $u\in[0,\fz)$ and $p\in(1,\fz)$.
\end{remark}

\smallskip

\noindent\textbf{Acknowledgements}\quad The authors would like to thank
Professor Huoxiong Wu for some useful discussions on the subject of this article.



\bigskip

\noindent Jin Tao, Dachun Yang (Corresponding author),
Wen Yuan and Yangyang Zhang

\smallskip

\noindent  Laboratory of Mathematics and Complex Systems
(Ministry of Education of China),
School of Mathematical Sciences, Beijing Normal University,
Beijing 100875, People's Republic of China

\smallskip

\noindent {\it E-mails}: \texttt{jintao@mail.bnu.edu.cn} (J. Tao)

\noindent\phantom{{\it E-mails:}} \texttt{dcyang@bnu.edu.cn} (D. Yang)

\noindent\phantom{{\it E-mails:}} \texttt{wenyuan@bnu.edu.cn} (W. Yuan)

\noindent\phantom{{\it E-mails:}} \texttt{yangyzhang@mail.bnu.edu.cn} (Y. Zhang)

\end{document}